\documentclass{article}

\usepackage{amsmath}
\usepackage{amssymb}
\usepackage{amsthm}
\usepackage[cmtip,all]{xy}
\usepackage{mathrsfs}  
\usepackage{enumerate}
\usepackage{tikz-cd}
\usepackage{breqn}
\usepackage[utf8]{inputenc}
\usepackage{stackrel}

\newtheorem{theorem}{Theorem}[section]
\newtheorem{lemma}[theorem]{Lemma}

\newtheorem{corollary}[theorem]{Corollary}
\newtheorem{proposition}[theorem]{Proposition}
\newtheorem{conjecture}{Conjecture}
\newtheorem{result}{Main Result}

\theoremstyle{definition}
\newtheorem{definition}[theorem]{Definition}

\theoremstyle{remark}
\newtheorem{remark}[theorem]{Remark}

\numberwithin{equation}{section}

\newtheorem{ex}[theorem]{Example}
\newtheorem*{notation}{Notation}
\newtheorem*{ack}{Acknowledgement}

\newcommand{\GG}{\mathscr{G}}
\newcommand{\OO}{\mathcal{O}} 
\newcommand{\tor}{\mathrm{tor}}
\newcommand{\Tr}{\mathrm{Tr}}
\newcommand{\rsw}{\operatorname{rsw}}
\newcommand{\sw}{\operatorname{Sw}}
\newcommand{\Res}{\operatorname{Res}}
\newcommand{\Spec}{\operatorname{Spec}}
\newcommand{\N}{\mathbb{N}}
\newcommand{\Z}{\mathbb{Z}}
\newcommand{\Q}{\mathbb{Q}}
\newcommand{\R}{\mathbb{R}}
\newcommand{\m}{\mathfrak{m}}

\DeclareMathOperator{\Coker}{Coker}

\begin{document}

\title{On ramification in transcendental extensions
	of local fields}

\author{Isabel Leal\thanks{Department of Mathematics, University of Chicago, 5734 S. University Avenue, Chicago, IL, 60637, USA. Electronic address: \texttt{isabel@math.uchicago.edu}}}

\date{}

\maketitle

\begin{abstract}
	Let $L/K$ be an extension of complete discrete valuation fields, and assume that the residue field of $K$ is perfect and of positive characteristic. The residue field of $L$ is not assumed to be perfect.
	
	In this paper, we prove a formula for the Swan conductor of the image of a character 
	$\chi \in H^1(K, \Q/\Z)$ in $H^1(L, \Q/\Z)$ for $\chi$ sufficiently ramified. Further, we define generalizations $\psi_{L/K}^{\mathrm{ab}}$ and  $\psi_{L/K}^{\mathrm{AS}}$ of the classical $\psi$-function and prove a formula for  $\psi_{L/K}^{\mathrm{ab}}(t)$ for sufficiently large $t\in \R$. 
\end{abstract}


\section{Introduction}

Let $K$ be a complete discrete valuation field. Classical ramification theory has extensively studied finite Galois extensions $L/K$ when the residue field of $K$ is perfect. Much progress has also been achieved when the residue field is no longer assumed to be perfect, such as K. Kato's generalization of the classical Swan conductor $\sw \chi\in \Z_{\geq0}$ for abelian characters $\chi:G(L/K)\to\Q/\Z$ (\cite{kato1989swan}) and A. Abbes and T. Saito's generalization of the upper ramification filtration $G(L/K)$ 
(\cite{abbessaito}).  Yet there are still many open questions, both when the residue field of $K$ is imperfect and when the extension $L/K$ is transcendental. 

Let $L/K$ be a finite Galois extension of complete discrete valuation fields with perfect residue fields.
Denote by  $e(L/K)$  the ramification index of $L/K$ and 
by   $D^\log_{L/K}$ the wild different of  $L/K$, i.e., $D^\log_{L/K}= D_{L/K}-e(L/K)+1$,  where $D_{L/K}$ is the different of $L/K$. 
It is classically known
that, if $\chi \in H^1(K,\Q/\Z)$ and $\chi_L$ is its image in $H^1(L,\Q/\Z)$, then, when $\sw \chi\gg0$, 
\begin{equation} \label{eq:psi}
\sw \chi_L = \psi_{L/K} (\sw\chi)= e(L/K) \sw \chi - D^\log_{L/K},
\end{equation}
where $\psi_{L/K}$ is the classical $\psi$-function 	 (see, for example, \cite{serre1979local}).

In this paper, we obtain a formula resembling (\ref{eq:psi}) for (possibly transcendental) extensions $L/K$ of complete discrete valuation fields when the residue field of $K$ is perfect but the residue field of $L$ is not necessarily perfect, and then define generalizations of the classical $\psi$-function. To be precise, we first prove the two following results, the first when $L$ is of equal positive characteristic and the second when $L$ is of mixed characteristic. Here $\hat{\Omega}^1_{\OO_L/\OO_K}(\log)$ denotes the completed $\OO_L$-module of relative differential forms with log poles, $\delta_\tor(L/K)$ the length of its torsion part, and $e_K$ the absolute ramification index of $K$. 
For a character $\chi\in H^1(L,\Q/\Z)$, $\sw \chi$ denotes Kato's Swan conductor of $\chi$ (defined in \cite{kato1989swan}).

\begin{result}[Theorem \ref{mainpositive}] \label{main re po}
	Let $L/K$ be an extension of complete discrete valuation fields of equal characteristic $p>0$. Assume that $K$ has perfect residue field and
	$\chi \in H^1(K,\Q/\Z)$ is such that
	\[\sw \chi> \dfrac{p}{p-1}\dfrac{\delta_\tor(L/K)}{e(L/K)}.\] 
	Denote  by $\chi_L$ its image in $H^1(L,\Q/\Z)$. Then
	\[\sw \chi_L= e(L/K)\sw\chi-\delta_\tor(L/K).\]
\end{result}
\begin{result}[Theorem \ref{main mixed}] \label{main re mix}
	Let $L/K$ be an  extension of complete discrete valuation fields of mixed characteristic. Assume that $K$ has perfect residue field of characteristic $p>0$ and  
	$\chi\in H^1(K,\Q/\Z)$ is such that \[\sw \chi \geq
	\dfrac{2e_K}{p-1}+\dfrac{1}{e(L/K)}+\left\lceil\dfrac{\delta_\tor(L/K)}{e(L/K)}\right\rceil.\] Denote by $\chi_L$ its image in $H^1(L,\Q/\Z)$. Then 
	\[\sw \chi_L = e(L/K) \sw \chi - \delta_\tor(L/K).\]
\end{result}

After proving these two main results, we relate this discussion to the $\psi_{L/K}$ function for $L/K$. More precisely,  we define two  $\psi$-functions $\psi_{L/K}^{\mathrm{AS}}$ and $\psi_{L/K}^{\mathrm{ab}}$ when $K$ has perfect residue field but $L$ has residue field not necessarily perfect. We then show that, in the classical case of finite $L/K$, both these definitions coincide with the classical $\psi_{L/K}$ function. Finally, we prove that we can regard our first two main theorems as formulas for $\psi_{L/K}^{\mathrm{ab}}(t)$ for $t\gg 0$:   
\begin{result}[Theorem \ref{main thm psi}] \label{main re psi}
	Let $L/K$ be an 
	extension of complete discrete valuation fields. Assume that $K$ has perfect residue field of characteristic $p>0$. Let $t\in\R_{\geq 0}$ be such that
	\[\begin{cases}
	t\geq \dfrac{2e_K}{p-1}+\dfrac{1}{e(L/K)}+\left\lceil\dfrac{\delta_\tor(L/K)}{e(L/K)}\right\rceil & \text{if $K$ is of characteristic $0$,} \\[.5cm] t> \dfrac{p}{p-1} \dfrac{\delta_\tor(L/K)}{e(L/K)} & \text{if $K$ is of characteristic $p$.}
	\end{cases}\]
	Then 
	\[\psi^{\mathrm{ab}}_{L/K}(t)= e(L/K)t-\delta_{\tor}(L/K). \]
\end{result} 

Our methods for the proof of Main Result \ref{main re po} 
differ greatly from those for the proof of Main Result \ref{main re mix}. 
In the equal characteristic case, we use Artin-Schreier-Witt theory. In the mixed characteristic case, we use M. Kurihara's exponential map (\cite{kurihara1998exponential}) and a modified version of higher dimensional local class field theory.  

We hope to apply the results of this paper to generalize a previous work (\cite{isabel1}) that studies the ramification of the action of the absolute Galois group $G_K$ on $H^1(U_{\overline{K}},\GG)$ (where $U$ is an open curve over $K$ and $\GG$ a smooth $\ell$-adic sheaf of rank $1$ on $U$) from the semi-stable case to a more general one.

The organization of this paper is the following: in Section \ref{sec:2}, we study the positive characteristic case and prove Main Result \ref{main re po}. In Section \ref{sec:3}, we introduce the discussion of the mixed characteristic case by studying the example of a two-dimensional local field whose last residue field is finite. In Section
\ref{sec gen mixed}, we study the general mixed characteristic case and prove Main Result \ref{main re mix}.  In Section \ref{sec5}, we define generalizations of the $\psi$-function when the residue field of $L$ is not necessarily perfect and prove Main Result \ref{main re psi}, which connects them with the other main results of this paper.

\begin{notation}
	Through this paper, for a complete discrete valuation field $K$, $\OO_K$ denotes its ring of integers, $\m_K$ the maximal ideal,  $\pi_K$ a prime element, and $G_K$ the absolute Galois group. Lowercase $k$  denotes the residue field of $K$, and $v_K$ the discrete valuation. We write $U_K^n=1+\m_K^n$.

	When we say that $K$ is a local field, we mean that  $K$ is a complete discrete valuation field with perfect (not necessarily finite) residue field. Similarly, when we say $K$ is a $q$-dimensional local field, we mean that there is a chain of fields  $K = K_q, K_{q-1}, \ldots, K_1, K_0$ such that, for each $1\leq i \leq q$, $K_{i}$ is a complete discrete valuation field with residue field $K_{i-1}$ and $K_0$ is a perfect field. When the last residue field $K_0$ is finite, we say that $K$ is a $q$-dimensional local field with  finite last residue field. 
	
	We write \[\hat{\Omega}^1_{\OO_K}(\log)= \varprojlim\limits_m \Omega^1_{\OO_K}(\log)/\m_K^m \Omega^1_{\OO_K}(\log),\] 
	where   
	\[
	\Omega^1_{\OO_K}(\log)= (\Omega_{\OO_K}^1\oplus(\OO_K\otimes_{\Z}K^\times))/(da - a\otimes a, \,  a \in \OO_K, \, a \neq 0).
	\]
	
	We shall denote by $P_\tor$ the torsion part of an abelian group $P$. Let $L/K$ an 
	extension of complete discrete valuation fields (of either mixed characteristic or positive characteristic $p>0$). Throughout this paper, $e(L/K)$ shall denote the ramification index of $L/K$ and $e_K$ the absolute ramification index of $K$. When $k$ is perfect,  $\delta_\tor(L/K)$ shall denote the length of $\left(\dfrac{\hat{\Omega}^1_{\OO_L}(\log)}{\OO_L\otimes_{\OO_K}\hat{\Omega}^1_{\OO_K}(\log)}\right)_\tor$.   
	
	The $r$-th Milnor $K$-group of $L$ shall be denoted by $K_r(L)$. 
	We denote by $U^nK_{r}(L)$ the subgroup of $K_{r}(L)$ generated by elements
	$\{a,b_1,\ldots,b_{r-1}\}$ where $a\in U_L^n$, $b_i\in L^\times$, and we write
	\[
	\hat{K}_{r}(L) = \varprojlim_n  K_r(L)/U^nK_r(L)	
	\]
	and
	\[
	U^n\hat{K}_{r}(L)= \varprojlim_{n'}U^nK_r(L)/U^{n'}K_r(L).
	\] 
	
	Following the notation in \cite{kato1989swan}, we write, for $A$ a ring over $\Q$ or a smooth ring over a field of characteristic $p>0$, and $n\neq 0$, 
	\[ H_n^q(A)=   H^q((\Spec A)_{\mathrm{et}}, \Z/n\Z(q-1)  )\]
	and
	\[H^q(A)= \varinjlim\limits_n   H_n^q(A).\]
\end{notation}

\section{Swan conductor in positive characteristic} \label{sec:2}

Let $L$ be complete discrete valuation field of equal characteristic $p>0$. In this section, we will study 
extensions $L/K$ where $K$ is a local field (and therefore $k$ is perfect). To be precise,   we shall 
show that, if $\chi\in H^1(K)$ has Swan conductor sufficiently large, then
\[\sw \chi_L= e\sw\chi-\delta_\tor(L/K),\]
where $\chi_L$ is the image of $\chi$ in $H^1(L)$ and $e=e(L/K)$.  For that goal, we will use valuations on differential forms and Witt vectors, as well as the notion of a Witt vector being ``best'', defined later. 

First of all, we review some concepts necessary for our discussion.\label{completed free}
By  completed free $\OO_L$-module with basis $\{e_\lambda\}_{\lambda\in\Lambda}$, we mean $\varprojlim\limits_m M/\m_L^mM$, where $M$ is the free $\OO_L$-module with basis $\{e_\lambda \}_{\lambda\in\Lambda}$. Write $L=l((\pi_L))$ for some  prime $\pi_L\in L$, where $l$ is the residue field of $L$. Let $\{b_\lambda\}_{\lambda\in \Lambda}$ be a lift of a  $p$-basis of $l$ to $\OO_L$. Then $\hat{\Omega}^1_{\OO_L}(\log)$ is the completed free $\OO_L$-module with basis $\{db_\lambda, d\log \pi_L:\lambda\in\Lambda\}$.  Write $\hat{\Omega}^1_L=L\otimes_{\OO_L}\hat{\Omega}^1_{\OO_L}(\log)$. 

Recall that, when $K$ is a local field of positive characteristic, 
$\hat{\Omega}^1_{\OO_{K}}(\log)$ is  free of rank one and, for an  
extension of complete discrete valuation fields $L/K$,    $\delta_\tor(L/K)$ is the length of the torsion part of $\hat{\Omega}^1_{\OO_L/\OO_K}(\log)$.

Denote by $W_s(L)$ the Witt vectors of length $s$. There is a homomorphism $d:W_s(L)\to \hat{\Omega}_{L}^1$ given by
\[a=(a_{s-1},\ldots, a_0)\mapsto  \sum_i a_i^{p^i-1}da_i.\]

\begin{remark} 
	In the literature, the operator $d:W_s(L)\to \hat{\Omega}_{L}^1(\log)$ is often denoted by $F^{s-1}d$. 
\end{remark}

We can define valuations on $\hat{\Omega}_{L}^1$ and $W_s(L)$ as follows. If $\omega \in 
\hat{\Omega}_{L}^1$ and $a \in W_s(L)$, let
\[
v_L^{\log} \omega = \sup\left\{n: \omega \in \pi_L^n\otimes_{\OO_L}\hat{\Omega}_{\OO_L}^1(\log)\right\},
\]
and
\[
v_L(a)=-\max_i\{- p^i v_L(a_i)\}=\min_i\{p^i v_L(a_i)\}.
\]

These valuations define increasing filtrations of   $\hat{\Omega}_{L}^1$ and $W_s(L)$ by the subgroups 
\[
F_n \hat{\Omega}_{L}^1 = \{\omega \in \hat{\Omega}_{L}^1:v_L^{\log} \omega \geq -n\}
\]
and
\[
F_n W_s(L) = \{a \in W_s(L):v_L(a) \geq -n\},
\]
respectively, where $n\in \Z_{\geq0}$. The latter filtration was defined by Brylinski in \cite{Brylinski1983}. 

By the theory of Artin-Schreier-Witt, there are isomorphisms
\[
W_s(L)/(F-1)W_s(L)\simeq H^1(L,\Z/p^s \Z),
\] 
where $F$ is the endomorphism of Frobenius. Kato defined in \cite{kato1989swan} the filtration $F_n H^1(L, \Z / p^s\Z)$
as the image of $F_n W_s(L)$ under this map. We recall that, for $\chi \in  H^1(L, \Z / p^s\Z)$, the Swan conductor $\sw \chi$ is the smallest $n$ such that $\chi \in F_n H^1(L, \Z / p^s\Z)$.

We shall now define what it means for a  Witt vector $a\in W_s(L)$ to be ``best'', as well as the notion of relevance length.

\begin{definition} \label{def best}
	Let	$a \in W_s(L)$, and $n$ be the smallest non-negative integer such that $a \in F_n W_s(L)$. 
	We say that $a$ is best if there is no $a'\in W_s(L)$ mapping to the same element as $a$ in $H^1(L, \Z / p^s\Z)$  such that $a'\in F_{n'}W_s(L)$ for some  non-negative integer $n'<n$.
\end{definition}

	When $v_L(a)\geq0$, $a$ is clearly best. When $v_L(a)<0$, $a$ is best if and only if  there are no $a', b \in W_s(L)$ satisfying
	\[
	a= a' + (F-1)b
	\]	
	and  $v_L(a)< v_L(a')$.
	
	Observe that  $a\in F_n W_s(L)\backslash F_{n-1} W_s(L)$ is best if and only if $n= \sw \chi$, where $\chi$ is the image of $a$ under $F_nW_s(L)\to H^1(L,\Z/p^s\Z)$. We remark that ``best $a$'' is not unique.

We shall start by deducing a simple criterion for determining when  $a$ is best.
When $s=1$ the characterization of ``best $a$'' is well-known: every $a\in \OO_L$ is best, and $a\in L \setminus \OO_L$ is best if and only if either $p\nmid v_L(a)$ or $p\mid v_L(a)$ but 
 $\bar{a}\notin l^p$, where $\bar{a}$ denotes the residue class of $a/\pi_L^{v_L(a)}$ for a prime element $\pi_L\in L$. In this section we will characterize best $a$ for arbitrary $s$. We shall prove that $a$ is best if and only if $a_i$ is best for some relevant position $i$, in the sense of the following definition.

\begin{definition}
	We shall say that the $i$-th position of $a$ is relevant if $v_L(a)=p^i v_L(a_i)$. Let $j=\max\{i:v_L(a)=p^i v_L(a_i)\}$. Then $j+1$ shall be called the relevance length of $a$.  
\end{definition}

\begin{lemma}\label{01}
	Let $a\in W_s(L)$ 
	be of negative valuation. We have $v_L(a)=v_L^\log (da)$ if and only if  
	there is some relevant position $k$ such that $v_L(a_k)=v_L^\log (da_k)$. 
\end{lemma}
\begin{proof}
	Let $I$ denote the subset of $\{0,\ldots, s-1\}$ consisting of $i$ such that the $i$-th position is relevant and $v(a_i)=v_L^\log (da_i)$.
	Let $j+1$ denote the relevance length of $a$. 
	We have
	\[
	da = \sum_{i\in I} a_i^{p^i-1} da_i + \sum_{i \notin I} a_i^{p^i-1}da_i.
	\]
	Clearly 
	\[
	v_L^\log\left(\sum_{i \notin I} a_i^{p^i-1}da_i\right)>v_L(a),
	\]
	so it is enough to prove that 
	\[
	v_L^\log\left( \sum_{i\in I} a_i^{p^i-1} da_i \right)= v_L(a)
	\]
	if  $I$ is nonempty. 
	
	Assume $I$ nonempty. Since the relevance length of $a$ is $j+1$, we get that $p^j\mid v_L(a)$. We have $v_L(a)=-np^j$ for some $n\in\N$. For each $i\in I$, we have $v_L(a_i)=-np^{j-i}$. Write $a_i= \pi_L^{-np^{j-i}} u_i$, where $u_i\in\OO_L$ is a unit. 
	
	Then 
	\[
	\sum_{i\in I} a_i^{p^i-1} da_i=  \pi_L^{-np^j}\left(\sum_{i\in I} u_i^{p^i-1} du_i -
	n u_j^{p^j} \frac{d\pi_L}{\pi_L}\right).
	\]
	If $p\nmid n$, then 
	\[
	v_L^\log\left( \sum_{i\in I} a_i^{p^i-1} da_i \right)= v_L(a).
	\]
	On the other hand, if $p\mid n$, 
	\[
	\sum_{i\in I} a_i^{p^i-1} da_i=  \pi_L^{-np^j}\sum_{i\in I} u_i^{p^i-1} du_i.
	\]
	Let $\bar{u}_i$ denote the image of $u_i$ in the residue field $l$. Then
	\[
	v_L^\log\left(\pi_L^{-np^j}\sum_{i\in I} u_i^{p^i-1} du_i\right)> v_L(a)
	\]
	if and only if
	\[
	\sum_{i\in I} \bar{u}_i^{p^i-1} d\bar{u}_i=0.
	\]
	If
	\[
	\sum_{i\in I} \bar{u}_i^{p^i-1} d\bar{u}_i=0,
	\]
	then, by repeatedly applying the Cartier operator, we see that 
	$\bar{u}_i\in l^p$ for every $i\in I$. This implies $v_L(a_i)<v_L^\log(da_i)$ for every $i\in I$, a contradiction. Hence we must have 
	\[
	v_L^\log(da)=v_L(a).  	\qedhere 
	\] 
\end{proof}

\begin{lemma}\label{0}
	Let $a\in W_s(L)$ be of negative valuation. Assume that $v_L(a) < v_L^\log(da)$ and the relevance length of $a$ is $1$. Then $a$ is not best.
\end{lemma}
\begin{proof}
	Since the relevance length is $1$, we have $v_L^\log(a^{p^i-1}_i da_i)\geq p^i v_L(a_i)>v_L(a_0)$ for $i>0$. 
	Therefore we must have $v_L(a_0)<v_L^\log(da_0)$, which implies that there exist $a_0', b_0 \in L$ such that $a_0 = a_0' +b_0^p - b_0$ and $v_L(a_0)<v_L(a_0')$. Let $a'=(0,\ldots, 0, a_0')$ and $b=(0,\ldots, 0, b_0)$. We have
	\[
	a = a' + (F-1)b,
	\]
	and $v_L(a)=v_L(a_0)<v_L(a')$, so $a$ is not best. 
\end{proof}

\begin{lemma}\label{1}
	Let $a\in W_s(L)$ be an element of negative valuation. 
	Assume that $v_L(a) < v_L^\log(da)$.  
	Then $a$ is not best.
\end{lemma}
\begin{proof}
	We shall prove by induction on the relevance length. The case in which $a$ has relevance length $1$ has been proven in Lemma \ref{0}. Assume now that $a$ has relevance length $j+1$.  
	
	From Lemma \ref{01}, $v(a_j)<v_L^\log (da_j)$, so there exist $a_j', b_j \in L$ such that $a_j = a_j' +b_j^p - b_j$ and $v_L(a_j)<v_L(a_j')$. Observe that $v_L(a_j)=pv_L(b_j)$. Let $b=(0,\ldots, 0, b_j,0,\ldots, 0)$ and $a'=a-(F-1)b$. Then
	\[
	a'=a-Fb+b=(a_{s-1},\ldots, a_{j+1}, a_j', \tilde{a}_{j-1}, \ldots \tilde{a}_0),
	\]
	where $p^i v_L(\tilde{a}_i)\geq v_L(a)$ for every $0\leq i \leq j-1$.
	
	We have two cases.  If $p^i v_L(\tilde{a}_i)> v_L(a)$ for all $0\leq i \leq j-1$, then $v_L(a')>v_L(a)$, so $a$ is not best. 
	
	On the other hand, if   $v_L(\tilde{a}_i)= v_L(a)$ for some $0\leq i \leq j-1$, then $a'$ has relevance length at most $j$ and $v_L(a')=v_L(a)$.		
	Further, $da'=da+db$. Since $v_L(a) < v_L^\log(da)$ and $v_L(a)= pv_L(b)\leq p v_L^\log(db)$, we have $v_L(a)<v_L^\log(da')$. Thus $v_L(a')<v_L^\log(da')$ and $a'$ is of relevance length at most $j$. By induction, $a'$ is not best, i.e., there are $a'', c\in W_s(L)$ such that 
	\[
	a'= a'' + (F-1)c,
	\]
	with $v(a')<v(a'')$. Then
	\[
	a= a'+ (F-1)b= a'' +(F-1)(b+c),
	\]
	with $v_L(a)<v_L(a'')$. Thus $a$ is not best. 
\end{proof}

\begin{theorem}\label{besta}
	Let $a\in W_s(L)$. The following conditions are equivalent:
	\begin{enumerate}[(i)]
		\item $a$ is best.
		\item There exists some relevant position $i$ such that $a_i$ is best in the sense of length one. 
		\item  $v_L(a)=v_L^\log(da)$. 
	\end{enumerate}
\end{theorem}
\begin{proof}
	Observe that, when $a$ has non-negative valuation,  $(i), (ii)$ and $(iii)$ are all simultaneously satisfied, so in the following   we assume $v_L(a)<0$. 
	
	$(ii)\Leftrightarrow (iii)$ by Lemma \ref{01}.
	
	Lemma \ref{1} proves $(i)\Rightarrow (iii)$.
	
	To prove $(iii)\Rightarrow (i)$, assume that $a$ is not best. Then there are $a',b\in W_s(L)$ such that $a=a' + (F-1)b$ and $v_L(a)<v_L(a')$.
	We have $p v_L^\log(db)\geq p v_L(b)= v_L(a)$, so both $v_L^\log(db)>v_L(a)$ and $v_L^\log(da')\geq v_L(a')>v_L(a)$. Since $da=da'-db$, we get that $v_L^\log(da)>v_L(a)$. 
\end{proof}

We shall now use the notion of ``best $a$'' to construct a homomorphism
$ F_n H^1(L, \Z / p^s \Z) \to F_n\hat{\Omega}^1_L/F_{\lfloor n/p\rfloor}\hat{\Omega}^1_L$ satisfying some useful properties. 
Given an element of $H^1(L, \Z/p^s \Z)$, it is easy to show the existence of a best $a\in W_s(L)$ in its preimage. We then have the following proposition:

\begin{proposition} 
	\label{prop:besta}
	\hspace{2em} 
	\begin{enumerate}[(i)] 
		\item There is a unique homomorphism
		\[
		\rsw: F_n H^1(L, \Z / p^s \Z) \to F_n\hat{\Omega}^1_L/F_{\lfloor n/p\rfloor}\hat{\Omega}^1_L,
		\]
		called refined Swan conductor, such that the composition 
		\begin{center}
			\begin{tikzcd}
				F_n W_s(L) \arrow[r] & F_n H^1(L, \Z / p^s \Z) \arrow[r] &   F_n\hat{\Omega}^1_L/F_{\lfloor n/p\rfloor}\hat{\Omega}^1_L
			\end{tikzcd}
		\end{center}
		coincides with
		\[
		d:  F_n W_s(L)  \to  F_n\hat{\Omega}^1_L/F_{\lfloor n/p\rfloor}\hat{\Omega}^1_L.
		\]
		\item For $\lfloor n/p \rfloor \leq m \leq n$, the induced map 
		\[
		\rsw:  F_n H^1(L, \Z / p^s \Z)/ F_m H^1(L, \Z / p^s \Z) \to F_n\hat{\Omega}^1_L/F_m\hat{\Omega}^1_L
		\]
		is injective.
	\end{enumerate}
\end{proposition}

\begin{proof}
	To prove assertion \textit{(i)}, define $\rsw$ as follows. Given an element $\chi \in F_n H^1(L, \Z / p^s \Z)$, take $a\in F_n W_s(L)$ such that $a$ is best and the image of $a$ is $\chi$. Then put $\rsw \chi = da$.
	
	We must show that  this map is well-defined. Let $a'\in F_n W_s(L)$ be another element that is best and maps to $\chi$. Then 
	\[
	a= a' + (F-1)b
	\]
	for some $b\in W_s(L)$. We get that  $p v_L^\log(db)\geq p v_L(b)\geq -n$, so $db \in F_{\lfloor n/p\rfloor}\hat{\Omega}^1_L$. Since $da= da' - db$, $da$ and $da'$ define the same class in $ F_n\hat{\Omega}^1_L/F_{\lfloor n/p\rfloor}\hat{\Omega}^1_L$. Uniqueness of the map is clear. 
	
	We shall now prove \textit{(ii)}. Let $\chi \in  F_n H^1(L, \Z / p^s \Z)$ such that $\rsw \chi \in F_m\hat{\Omega}^1_L$. Take $a\in  F_n W_s(L)$ that is best and such that $da=\rsw \chi$. Since $a$ is best, we have
	\[v^\log_L(\rsw\chi)=v^\log_L(da)=v_L(a)\geq -m,\]
	so $a\in F_m W_s(L)$. It follows that $\chi \in  F_m H^1(L, \Z / p^s \Z)$.
\end{proof}

\begin{remark}
	Related results were obtained by Y. Yatagawa in \cite{yatagawa2016equality}, where the author compares the non-logarithmic filtrations of Matsuda (\cite{matsuda1997swan}) and Abbes-Saito (\cite{abbessaito}) in positive characteristic.
\end{remark}
\begin{remark}
	Our refined Swan conductor $\rsw$ is a refinement of the refined Swan conductor defined by K. Kato in \cite[\S 5] {kato1989swan}.
\end{remark}

Let $L/K$ be an 
extension of complete discrete valuation fields of positive characteristic $p>0$, and assume that $K$ has perfect residue field $k$. Let $\chi\in H^1(K)$  and $\chi_L$ its image in $H^1(L)$.  
We shall now use Proposition \ref{prop:besta} 
to compute the Swan conductor of $\chi_L$. We will need the following lemma:  

\begin{lemma}\label{differential} 
	Let $L/K$ be an 
	extension of complete discrete valuation fields of equal characteristic $p>0$. Write $e=e(L/K)$ and assume that $k$ is perfect. 
	
	Let $\omega \in \hat{\Omega}^1_K$, and $\omega_L$ be the image of $\omega$ in $\hat{\Omega}^1_{L}$.
	Then
	\[
	v_L^{\log}(\omega_L)= e v_K^{\log}(\omega) +\delta_\tor(L/K).
	\]		
\end{lemma}
\begin{proof}
	Since the residue field $k$ of $K$ is perfect, $\hat{\Omega}^1_{\OO_K}(\log)=\OO_K \frac{d\pi_K}{\pi_K}$. Let $\{b_\lambda\}_{\lambda\in\Lambda}$ be a lift of a $p$-basis of $l$ to $\OO_L$, so that $\hat{\Omega}^1_{\OO_L}(\log)$ is the completed free module with basis $\{db_\lambda, d\log\pi_L:\lambda \in \Lambda\}$.  Write $\frac{d\pi_K}{\pi_K}= \sum \alpha_\lambda db_\lambda + \alpha d\log \pi_L$, where $\alpha_\lambda, \alpha \in \OO_L$. Then \[\delta_\tor(L/K)=\min\{\{v_L(\alpha)\}\cup \{v_L(\alpha_\lambda): \lambda\in \Lambda\}\}= v_L^\log\left(\frac{d\pi_K}{\pi_K}\right).\] Writing $\omega = \gamma \frac{d\pi_K}{\pi_K}$ for some $\gamma \in K$, we see that 
	\[
	v_L^\log(\omega) = v_L(\gamma) + v_L^\log\left(\frac{d\pi_K}{\pi_K}\right)= ev_K(\gamma)+\delta_\tor(L/K) = e v_K^\log(\omega) + \delta_\tor(L/K).\qedhere
	\] 
\end{proof}

\begin{theorem} \label{mainpositive}
	Let $L/K$ be an 
	extension of complete discrete valuation fields of equal characteristic $p>0$. Assume that $K$ has perfect residue field. 
	
	Denote by $e(L/K)$ the ramification index of $L/K$. 
	Assume that 	  $\chi \in H^1(K)$ is such that
	\[\sw \chi> \dfrac{p}{p-1} \dfrac{\delta_\tor(L/K)}{e(L/K)}.\] 
	Let $\chi_L$ be its image in $H^1(L)$. Then
	\[\sw \chi_L= e(L/K)\sw\chi-\delta_\tor(L/K).\]
\end{theorem}
\begin{proof}
	Write $e=e(L/K)$. It is enough to show that, for a character $\chi \in H^1(K, \Z / p^s \Z)$  corresponding to the Artin-Schreier-Witt  equation $(F-1)X= a$, we have that,  	
	if $\sw \chi> p(p-1)^{-1}e^{-1}\delta_\tor(L/K)$, then
	\[\sw \chi_L= e\sw\chi-\delta_\tor(L/K).\]
	
	To simplify notation, write $n=\sw \chi$, $\delta_\tor = \delta_\tor(L/K)$. The case $e=1$ is simple, so we assume $e>1$. Since  $\sw \chi> p(p-1)^{-1}e^{-1}\delta_\tor(L/K)$, we have that $\frac{en}{p}< en-\delta_\tor$, so $\lfloor\frac{en}{p}\rfloor \leq en-\delta_\tor -1$.  From that,  Theorem \ref{besta}, and Lemma \ref{differential}, we get that the   diagram 
	\begin{center}
		\begin{tikzcd}
			F_n H^1(K, \Z / p^s \Z)/ F_{n-1} H^1(K, \Z / p^s \Z)  \arrow[r]\arrow[d]&  F_n\hat{\Omega}^1_K/ F_{n-1}\hat{\Omega}^1_K \arrow[d]\\ F_{en} H^1(L, \Z / p^s \Z)/ F_{en-\delta_\tor-1} H^1(L, \Z / p^s \Z)  \arrow[r] &  F_{en}\hat{\Omega}^1_L/F_{en-\delta_\tor -1}\hat{\Omega}^1_L
		\end{tikzcd}
	\end{center}
	commutes, and the horizontal arrows are injective. Thus
	\[\sw \chi_L= e(L/K)\sw\chi-\delta_\tor(L/K).\qedhere\]
\end{proof}

\section{The example of a two-dimensional local field of  mixed characteristic with finite last residue field}
\label{sec:3}

In Section \ref{sec:2}, we proved Main Result \ref{main re po}. We shall now focus on 
proving Main Result \ref{main re mix}. Let $L/K$ be an extension of complete discrete valuation fields of mixed characteristic, and assume that $K$ has perfect residue field. We will 
show that, if $\chi\in H^1(K)$ has Swan conductor sufficiently large, then
\[\sw \chi_L= e\sw\chi-\delta_\tor(L/K),\]
where $\chi_L$ is the image of $\chi$ in $H^1(L)$ and $e=e(L/K)$ is the ramification index of $L/K$.

The proof of this result is based on two key ideas: the commutativity of a diagram of the form
\begin{center}
	\begin{tikzcd}
		\m_L^{en'-\delta_\tor(L/K)}	\hat{\Omega}^{q-1}_{\OO_{L}}(\log) \arrow[r, "\exp_\eta"] \arrow[d, "\Res_{L/K}"] 
		& \hat{K}_q(L) \arrow[d, "\Res_{L/K}"] \\
		\m_K^{n'} \arrow[r, "\exp_\eta"]
		& K^\times
	\end{tikzcd}
\end{center}
and a modified version of higher dimensional local class field theory.  In order to facilitate comprehension and illustrate the main ideas, in the present section we will consider, in a brief and expository way, the special case in which $L$ is a two-dimensional local field with finite last residue field. In this special case, the second key idea is simpler, since we can use two-dimensional local class field theory without any modification. In Section \ref{sec gen mixed} we consider the general case in which $L$ is a complete discrete valuation field of mixed characteristic.  

Through this section, we let $L$ be a two-dimensional local field of mixed characteristic with residue field $l$ of characteristic $p>0$, and $K\subset L$ a one-dimensional local field with finite residue field $k$. 

As a consequence of \cite{morrow2010explicit}, there is a residue homomorphism
\[
\Res_{L/K}: 	\hat{\Omega}^1_{\OO_L}\to \OO_{K}
\]
which induces
\[
\Res_{L/K}: 	\hat{\Omega}^1_{\OO_L}(\log)\otimes_{\OO_L}L\to K.		
\]
\begin{ex}
	When $L=K\{\{T\}\}$ (see page \pageref{kx}), 
	\[
	\Res_{L/K}\left(\sum\limits_{i=-\infty}^\infty a_i T^i\frac{dT}{T}\right)=a_0.
	\]
\end{ex}

From  \cite{kurihara1998exponential}, if $\eta\in\OO_L$ is such that $v_L(\eta)\geq \dfrac{2e_L}{p-1}+1$, there exists  an exponential map 
\[	\exp_\eta: 	\hat{\Omega}^1_{\OO_{L}}(\log) \to \hat{K}_2(L).  \]
This map is used in the following theorem, which is the first key step in the proof of the main result for the special case of a two-dimensional local field with finite last residue field. Its proof is omitted due to similarity with that of Theorem \ref{key2}.

\begin{theorem}\label{the:com}
	Let $L$ be a two-dimensional local field of mixed characteristic and with finite last residue field, and $K\subset L$  a local field. Write $e= e(L/K)$. Let $\eta \in \OO_K$ be such that 
	\[n=v_K(\eta)\geq
	\frac{2e_K}{p-1}+\frac{1}{e}.\]
	Then, if $n'\in \N$  satisfies
	\[
	n'\geq \frac{\delta_\tor(L/K)}{e},
	\] we have a commutative diagram 
	\begin{center}
		\begin{tikzcd}
			\m_L^{en'-\delta_\tor(L/K)}	\hat{\Omega}^1_{\OO_{L}}(\log) \arrow[r, "\exp_\eta"] \arrow[d, "\Res_{L/K}"] 
			& \hat{K}_2(L) \arrow[d, "\Res_{L/K}"] \\
			\m_K^{n'} \arrow[r, "\exp_\eta"]
			& K^\times
		\end{tikzcd}
	\end{center}
	where the right vertical arrow is the residue homomorphism from $K$-theory defined in \cite{kato1983residue} and the top and bottom horizontal maps are, respectively, the exponential maps $\exp_{\eta,2}$ and $\exp_{\eta,1}$ defined in \cite{kurihara1998exponential}.
\end{theorem}

We observe that $\m_L^{en'-\delta_\tor(L/K)}	\hat{\Omega}^1_{\OO_{L}}(\log)\to \m_K^{n'}$ in the diagram above is surjective (see Proposition \ref{cor:2}) and the images of \[\exp_\eta:\m_L^{en'-\delta_\tor(L/K)}	\hat{\Omega}^1_{\OO_{L}}(\log)\to \hat{K}_2(L)\] and \[\exp_\eta:\m_K^{n'}\to K^\times\] are, respectively, $U^{e(n+n')-\delta_\tor(L/K)}\hat{K}_2(L)$ and $U_K^{n+n'}$ (see Lemma \ref{lemma:exp}).

Theorem \ref{the:com}  is then combined with two-dimensional local class field theory to prove the main result in the particular case of a two-dimensional local field of mixed characteristic with finite last residue field:
\begin{theorem} \label{the:maintwo}
	Let $L$ be a two-dimensional local field of mixed characteristic with finite last residue field, and $K\subset L$ be a local field. Assume that  $\chi\in H^1(K)$ is  such that \[\sw \chi \geq
	\dfrac{2e_K}{p-1}+\dfrac{1}{e(L/K)}+ \left\lceil\dfrac{\delta_\tor(L/K)}{e(L/K)}\right\rceil.\] Denote by $\chi_L$ its image in $H^1(L)$. Then 
	\[\sw \chi_L = e(L/K) \sw \chi - \delta_\tor(L/K).\]
\end{theorem}
\begin{proof}
	Write $e= e(L/K)$.
	Let $n'=\left\lceil\frac{\delta_\tor(L/K)}{e}\right\rceil$ and $n= \sw \chi - n'$. Pick $\eta \in \OO_K$ with $v_K(\eta)=n$. 
	By two-dimensional local  class field theory, the diagram 
	\begin{center}
		\begin{tikzcd}
			\hat{K}_2(L) \arrow[d, "\Res_{L/K}"] \arrow[r] & G_L^{\mathrm{ab}} \arrow[d]  \\
			K^\times \arrow[r] & G_K^{\mathrm{ab}}
		\end{tikzcd}
	\end{center}
	commutes. 
	Together with Theorem \ref{the:com}, this gives us a commutative diagram 
	\begin{center}
		\begin{tikzcd}
			\m_L^{en'-\delta_\tor(L/K)}	\hat{\Omega}^1_{\OO_{L}}(\log) \arrow[r, "\exp_\eta"] \arrow[d, "\Res_{L/K}"] 
			& \hat{K}_2(L) \arrow[d, "\Res_{L/K}"] \arrow[r] & G_L^{\mathrm{ab}} \arrow[d]  \\
			\m_K^{n'} \arrow[r, "\exp_\eta"]
			& K^\times \arrow[r] & G_K^{\mathrm{ab}} 
		\end{tikzcd}
	\end{center}
	
	From Proposition \ref{cor:2}, the left vertical arrow is surjective.  We know that $\sw \chi =m$ if and only if $\chi$ kills $U_K^{m+1}$ but not  $U_K^{m}$, and $\sw \chi_L =m$ if and only if $\chi_L$ kills $U^{m+1}\hat{K}_2(L)$ but not  $U^{m}\hat{K}_2(L)$ (see the proof of Proposition \ref{most} for details). Then it follows from the commutative diagram above and Lemma \ref{lemma:exp}
	that 
	\[\sw \chi_L = e (n'+n) - \delta_\tor(L/K)=e \sw \chi - \delta_\tor(L/K).\qedhere\]
\end{proof}

As a guide for Section \ref{sec gen mixed}, we will use Theorem \ref{the:maintwo} to get the same result for a complete discrete valuation field of mixed characteristic $L$ which has residue field that is a function field in one variable over a finite field. In Section \ref{sec gen mixed}, Proposition \ref{most} will be used to obtain Theorem \ref{main mixed} in an analogous way.

\begin{corollary}
	Let $L$ be a complete discrete valuation field of mixed characteristic, and $K\subset L$ be a local field.  Assume that the residue field $l$ of $L$ is a function field in one variable over the finite residue field $k$ of $K$.
	
	Assume that  $\chi\in H^1(K)$ is  such that \[\sw \chi \geq
	\dfrac{2e_K}{p-1}+\dfrac{1}{e(L/K)}+\left\lceil\dfrac{\delta_\tor(L/K)}{e(L/K)}\right\rceil.\] Denote by $\chi_L$ its image in $H^1(L)$. Then 
	\[\sw \chi_L = e(L/K) \sw \chi - \delta_\tor(L/K).\]
\end{corollary}
\begin{proof}
	It is sufficient to prove that this case can be reduced to that of  a two-dimensional local field with finite last residue field.
	
	Since $l$ is a function field in one variable over $k$, $l$ is a finite separable extension of $k(T)$ for some transcendental element $T$. Then there is an embedding of $l$ into a finite separable extension $E$ of $k((T))$. Note that $\{T\}$ is a $p$-basis for both $l$ and $E$. Then there is a complete discrete valuation field $L(E)$ which is an extension of $L$ satisfying $\OO_L\subset \OO_{L(E)}$, $\OO_{L(E)}\m_L= \m_{L(E)}$,  and the residue field of $L(E)$ is isomorphic to $E$ over $l$.
	
	From \cite[Lemma 6.2]{kato1989swan}, we get that  $\sw \chi_{L(E)}=\sw \chi_{L}$. Further, since $E$ is a one-dimensional local field, $L(E)$ is a two-dimensional local field.
	Finally, since $E$ and $l$ have the same $p$-basis $\{T\}$, and $\pi_L$ is a prime for both $L$ and $L(E)$, the map $\OO_{L(E)}\otimes_{\OO_L} \hat{\Omega}^1_{\OO_{L}}(\log)\to \hat{\Omega}^1_{\OO_{L(E)}}(\log)$ is an isomorphism and we get  $\OO_{L(E)}\otimes_{\OO_L} \hat{\Omega}^1_{\OO_{L}}(\log)_\tor\simeq \hat{\Omega}^1_{\OO_{L(E)}}(\log)_\tor$. Therefore, by definition, $\delta_\tor(L(E)/K)=\delta_\tor(L/K)$.
	
	Thus it is sufficient to prove that
	\[\sw \chi_{L(E)} = e(L(E)/K) \sw \chi - \delta_\tor(L(E)/K),\]
	which follows from Theorem \ref{the:maintwo}. 
\end{proof}

\section{Swan conductor in the general mixed characteristic case} \label{sec gen mixed}

In this section, we shall generalize the results of the previous section to the more general case in which $L$ is any complete discrete valuation field of mixed characteristic.  
We start by briefly reviewing some necessary background and proving some preliminary results.

Let $L$ be a complete discrete valuation field of mixed characteristic. Let $B$ be a lift of a $p$-basis of the residue field $l$ to $\OO_L$. Write $\{e_\lambda\}_{\lambda\in\Lambda} = \{db : b\in B\} \cup \{d\log \pi_L\}$. The $\OO_L$-module $\hat{\Omega}^1_{\OO_L}(\log)$ has the structure 
\[
\hat{M}\oplus \OO_L/\m_L^a\OO_L
\]
for some $a\in\Z_{\geq0}$ (see \cite[Lemma 1.1]{kurihara1987two} and   \cite[4.3]{kato2010modulus}). Here $\hat{M}$ is the completed free $\OO_L$-module with basis $\{e_\lambda \}_{\lambda \in \Lambda-\{\mu \}}$, i.e., $\hat{M}= \varprojlim\limits_{m} M/\m_L^m M$ where $M$ is the free $\OO_L$-module with basis $\{e_\lambda \}_{\lambda \in \Lambda-\{\mu \}}$ for some $\mu \in \Lambda$.

We have, from \cite[Theorem 0.1]{kurihara1998exponential}, the existence of an exponential map 
\[
\exp_{\eta,r+1}:	\hat{\Omega}^{r}_{\OO_{L}}(\log) \to  \hat{K}_{r+1}(L) 
\]
when  $\eta\in\OO_L$ satisfies  
\[v_L(\eta)\geq \frac{2e_L}{p-1}+1.\]
This exponential map satisfies
\[
a \frac{db_1}{b_1}\wedge\cdots\wedge\frac{db_{r}}{b_{r}} \mapsto \{\exp(\eta a),b_1,\ldots,b_{r}\}
\]
for $a\in \OO_L$, $b_i\in \OO_L^\times$. We shall denote $\exp_{\eta,r+1}$ simply by $\exp_{\eta}$ through this paper.
\begin{remark}
	More precisely, in \cite{kurihara1998exponential}, M. Kurihara proved the existence of an exponential map 
	\[
	\exp_{\eta,r+1}:	\hat{\Omega}^{r}_{\OO_{L}} \to  \hat{K}_{r+1}(L) 
	\]
	when
	$\eta\in\OO_L$ satisfies  
	\[v_L(\eta)\geq \frac{2e_L}{p-1}.\] 
	Considering the existence of a map $\hat{\Omega}^r_{\OO_{L}}(\log)\to	\hat{\Omega}^r_{\OO_{L}}$ satisfying the commutative diagram 
	\begin{center}
		\begin{tikzcd}
			\hat{\Omega}^r_{\OO_{L}}(\log) \arrow[r, "\pi_L"] 
			& \hat{\Omega}^r_{\OO_{L}}\\
			\hat{\Omega}^r_{\OO_{L}} \arrow[ur, "\pi_L"]\arrow[u] 
		\end{tikzcd}
	\end{center}
	we can define, for 
	\[v_L(\eta)\geq \frac{2 e_L}{p-1}+1,\]
	an exponential map
	\[
	\exp^{\log}_{\eta,r+1}:	\hat{\Omega}^{r}_{\OO_{L}}(\log) \to  \hat{K}_{r+1}(L) 
	\]
	by taking the composite
	\[\exp^{\log}_{\eta,r+1}= \exp_{\frac{\eta}{\pi_L},r+1}\circ \pi_L. \]
	Through this paper, we omit the superscript $\log$ when we write this exponential map.  
\end{remark}

\begin{lemma} Let $L$ be a complete discrete valuation field of mixed characteristic, with residue field $l$ of characteristic $p>0$.
	Assume that $\eta \in\OO_L$ satisfies  
	\[n = v_L(\eta)\geq \frac{2e_L}{p-1}+1.\]
	Then the image of the exponential map
	\[
	\exp_\eta:\m_L^{n'}\hat{\Omega}^{r}_{\OO_{L}}(\log) \to  \hat{K}_{r+1}(L) 
	\]
	is $U^{n+n'}\hat{K}_{r+1}(L)$. \label{lemma:exp}
\end{lemma}
\begin{proof} Let $a\in \m_L^{n'}$, $b_i\in \OO_L^\times$. Observe that, from the definition of the exponential map and \cite[Proposition 3.2]{kurihara1998exponential},
	\[
	a \frac{d\pi_L}{\pi_L}\wedge \frac{db_1}{b_1}\wedge\cdots\wedge\frac{db_{r-1}}{b_{r-1}} \mapsto \{\exp(pa\eta),\pi_L,  b_1, \ldots, b_{r-1}\}
	\]
	and
	\[
	a   \frac{db_1}{b_1}\wedge\cdots\wedge\frac{db_{r}}{b_{r}} \mapsto \{\exp(a\eta), b_1, \ldots, b_{r}\}.
	\]
	
	Then the image is contained in   $U^{n+n'}\hat{K}_{r+1}(L)$. Let $\tilde{n}\geq n+n'$. Observe that the maps 
	\[
	\frac{\m_L^{\tilde{n}}}{\m_L^{\tilde{n}+1}} \otimes \Omega_{\OO_L}^r(\log)\to U^{\tilde{n}}K_{r+1}(L)/U^{\tilde{n}+1}K_{r+1}(L)
	\]
	given by 
	\[
	\alpha\otimes \beta   \frac{db_1}{b_1}\wedge\cdots\wedge\frac{db_{r}}{b_{r}} \mapsto \{1+\alpha\beta,b_1,\ldots, b_{r}\},
	\]
	where $\alpha\in\m_L^{\tilde{n}}$, $\beta \in \OO_L, b_i \in L^\times$,  are surjective. Passing to the limit, we get that $\exp_\eta:\m_L^{n'}\hat{\Omega}^r_{\OO_L}(\log)\to U^{n+n'}\hat{K}_{r+1}(L)$ is surjective.
\end{proof}

We shall now construct some tools and intermediate steps necessary for the obtainment of the main result. For an extension of complete discrete valuation fields of mixed characteristic $L/K$, where $k$ is not necessarily perfect, denote by	$\delta_\tor(L/K)$  the length of 
\[\dfrac{\hat{\Omega}^1_{\OO_L}(\log)_{\mathrm{tor} }}{\OO_L\otimes_{\OO_K}\hat{\Omega}^1_{\OO_K}(\log)_{\tor }}.\]
\begin{remark}
	When $k$ is perfect,   the $\OO_K$-module 
	$\hat{\Omega}^1_{\OO_{K}}(\log)$ is a torsion module, and therefore   $\delta_\tor(L/K)$ is simply the length of 
	\[
	\left(\hat{\Omega}^1_{\OO_L/\OO_K}(\log)\right)_\tor,
	\]
	which coincides with the definition of  $\delta_\tor(L/K)$ introduced previously.
\end{remark}

We have the following property:

\begin{lemma} \label{lemma:trace}
	Let $L/M$ be a finite extension of complete discrete valuation fields of characteristic zero. Assume that the residue field $l$ of $L$ has characteristic $p>0$ and $[l:l^p]=p^{r}$. Write $e= e(L/M)$. 
	Then
	\[
	\Tr_{L/M}\left(\m_L^{en-\delta_\tor(L/M)}\dfrac{\hat{\Omega}^{r}_{\OO_L}(\log)}{\hat{\Omega}^{r}_{\OO_L}(\log)_\tor}\right)=\m_M^n\dfrac{\hat{\Omega}^{r}_{\OO_M}(\log)}{\hat{\Omega}^{r}_{\OO_M}(\log)_\tor}
	\] 		
	and
	\[
	\Tr_{L/M}\left(\m_L^{en-\delta_\tor(L/M)+1}\dfrac{\hat{\Omega}^{r}_{\OO_L}(\log)}{\hat{\Omega}^{r}_{\OO_L}(\log)_\tor}\right)= \m_M^{n+1}\dfrac{\hat{\Omega}^{r}_{\OO_M}(\log)}{\hat{\Omega}^{r}_{\OO_M}(\log)_\tor}
	\] 		
	for every integer $n$. 
\end{lemma}
\begin{proof}
	We shall prove the first equality. 
	Let $\delta(L/M)$ be the length of 
	the $\OO_L$-module
	$\hat{\Omega}^1_{\OO_L/\OO_M}(\log)$.
	Observe that $\dfrac{\hat{\Omega}^1_{\OO_L}(\log)}{\hat{\Omega}^1_{\OO_L}(\log)_\tor}$ and $\dfrac{\hat{\Omega}^1_{\OO_M}(\log)}{\hat{\Omega}^1_{\OO_M}(\log)_\tor}$ are free of rank $r$. We  have an exact sequence
	\begin{center}
		\begin{tikzcd}[column sep=8pt]
			0 \arrow[r] & \OO_L\stackbin[\OO_{M}]{}{\otimes} \dfrac{\hat{\Omega}^1_{\OO_M}(\log)}{\hat{\Omega}^1_{\OO_M}(\log)_\tor} \arrow[r]&\dfrac{\hat{\Omega}^1_{\OO_L}(\log)}{\hat{\Omega}^1_{\OO_L}(\log)_\tor}\arrow[r]& \dfrac{\dfrac{\hat{\Omega}^1_{\OO_L}(\log)}{\hat{\Omega}^1_{\OO_L}(\log)_\tor}}{
				\OO_L\stackbin[\OO_{M}]{}{\otimes}\dfrac{\hat{\Omega}^1_{\OO_M}(\log)}{\hat{\Omega}^1_{\OO_M}(\log)_\tor}}\arrow[r] & 0.
		\end{tikzcd}
	\end{center} 
	Since the length of
	\[
	\dfrac{\hat{\Omega}^1_{\OO_L}(\log)}{\hat{\Omega}^1_{\OO_L}(\log)_\tor}\Bigg/{
		\left(\OO_L\otimes_{\OO_M}\dfrac{\hat{\Omega}^1_{\OO_M}(\log)}{\hat{\Omega}^1_{\OO_M}(\log)_\tor}\right)}
	\] 
	is $\delta(L/M)-\delta_\tor(L/M)$, 	we have that  
	the length of
	\[
	\dfrac{\hat{\Omega}^{r}_{\OO_L}(\log)}{\hat{\Omega}^{r}_{\OO_L}(\log)_\tor}\Bigg/{
		\left(\OO_L\otimes_{\OO_M}\dfrac{\hat{\Omega}^{r}_{\OO_M}(\log)}{\hat{\Omega}^{r}_{\OO_M}(\log)_\tor}\right)}
	\]
	is also $\delta(L/M)-\delta_\tor(L/M)$.
	Since 
	$\dfrac{\hat{\Omega}^{r}_{\OO_L}(\log)}{\hat{\Omega}^{r}_{\OO_L}(\log)_\tor}$ and 
	$\dfrac{\hat{\Omega}^{r}_{\OO_M}(\log)}{\hat{\Omega}^{r}_{\OO_M}(\log)_\tor}$ are both free of rank one, we have
	\[
	\dfrac{\hat{\Omega}^{r}_{\OO_L}(\log)}{\hat{\Omega}^{r}_{\OO_L}(\log)_\tor}=\m_L^{\delta_\tor(L/M) - \delta(L/M)} 
	\dfrac{\hat{\Omega}^{r}_{\OO_M}(\log)}{\hat{\Omega}^{r}_{\OO_M}(\log)_\tor}.
	\] 
	Therefore
	\begin{gather*}	
	\Tr_{L/M}\left(\m_L^{en-\delta_\tor(L/M)}\dfrac{\hat{\Omega}^{r}_{\OO_L}(\log)}{\hat{\Omega}^{r}_{\OO_L}(\log)_\tor}\right)= \\ \Tr_{L/M}\left(\m_L^{en-\delta_\tor(L/M)}\m_L^{\delta_\tor(L/M) - \delta(L/M)}\dfrac{\hat{\Omega}^{r}_{\OO_M}(\log)}{\hat{\Omega}^{r}_{\OO_M}(\log)_\tor}\right)=\\
	\Tr_{L/M}\left(\m_L^{en-\delta(L/M)}\right)\dfrac{\hat{\Omega}^{r}_{\OO_M}(\log)}{\hat{\Omega}^{r}_{\OO_M}(\log)_\tor}.\end{gather*} 
	Let $\tilde{\delta}(L/M)$ be the length of the $\OO_L$-module  $\hat{\Omega}^1_{\OO_L/\OO_M}$. Since 
	\[
	\Tr_{L/M}\left(\m_L^{e(n+1)-\tilde{\delta}(L/M)-1}\right)=\m_M^n
	\]
	and $\tilde{\delta}(L/M)=\delta(L/M) +e -1$, we get
	\[
	\Tr_{L/M}\left(\m_L^{en-\delta(L/M)}\right)=\m_M^n.
	\]
	Hence
	\[
	\Tr_{L/M}\left(\m_L^{en-\delta_\tor(L/M)}\dfrac{\hat{\Omega}^{r}_{\OO_L}(\log)}{\hat{\Omega}^{r}_{\OO_L}(\log)_\tor}\right)=\m_M^n\dfrac{\hat{\Omega}^{r}_{\OO_M}(\log)}{\hat{\Omega}^{r}_{\OO_M}(\log)_\tor}.
	\] 	
	
	The second equality is obtained similarly.
\end{proof}

We shall now review $q$-dimensional local fields (for more on this subject, see \cite{zhukov2000, morrow2010explicit}). Subsequently, we shall use $q$-dimensional local fields to construct some residue maps. 

Let $K$ be a complete discrete valuation field. The field $K\{\{T\}\}$  is defined as the set
\begin{align*}
K&\{\{T\}\}=\\&\left\{\sum\limits_{i=-\infty}^\infty a_i T^i : a_i \in K, \ \inf v_K(a_i) > - \infty, \text{ and } v_K(a_i)\to \infty \text{ as } i \to - \infty\right\}
\end{align*}\label{kx}
with addition and multiplication as follows:
\[
\sum\limits_{i=-\infty}^\infty a_i T^i + \sum\limits_{i=-\infty}^\infty b_i T^i = \sum\limits_{i=-\infty}^\infty (a_i+b_i) T^i
\]
and
\[
\sum\limits_{i=-\infty}^\infty a_i T^i \sum\limits_{i=-\infty}^\infty b_i T^i = \sum\limits_{i=-\infty}^\infty\sum\limits_{j=-\infty}^\infty a_jb_{i-j}  T^i.
\]
We can define a discrete valuation on $K\{\{T\}\}$ by setting 
\[
v_{K\{\{T\}\}}\left(\sum\limits_{i=-\infty}^\infty a_i T^i\right)= \min v_K(a_i).
\]
Endowed with this valuation, $K\{\{T\}\}$ becomes a complete discrete valuation field with residue field $k((T))$. 

When $K$ is a local field, the field
\[
K\{\{T_1\}\}\cdots\{\{T_{m}\}\}((T_{m+1}))\cdots((T_{q-1})),
\]
where $1\leq m \leq q-1$, is a $q$-dimensional local field. Fields of this form are called standard
$q$-dimensional local fields.  

We shall now make the constructions necessary for defining
a residue map
\[
\Res_{L/K}:\hat{\Omega}^{q-1}_{\OO_L} (\log)\to \OO_{K}
\]
for a finite extension $L$ of 
$K\{\{T_1\}\}\cdots\{\{T_{q-1}\}\}$, where $K$ is a local field of mixed characteristic.

\begin{definition} Let $K$ be a complete discrete valuation field and $L_0=K, L_1=K\{\{T_1\}\}, \ldots$, $L= L_{q-1}= K\{\{T_1\}\}\cdots\{\{T_{q-1}\}\}$. Define 
	\[
	c_{L_i/L_{i-1}}:L_i\to L_{i-1}
	\]
	by 
	\[
	c_{L_i/L_{i-1}}\left(\sum_{k\in\Z}a_k T_i^k\right)=a_0.
	\]
	Then define $c_{L/K}= c_{L_1/L_0}\circ\cdots\circ c_{L_{q-1}/L_{q-2}}$.
\end{definition}

\begin{definition} \label{standard residue} Let $K$ be a local field of mixed characteristic and $L_0=K, L_1=K\{\{T_1\}\}, \ldots$, $L= L_{q-1}= K\{\{T_1\}\}\cdots\{\{T_{q-1}\}\}$. Define the residue map $\Res_{L_i/L_{i-1}}$ as the composition \[\hat{\Omega}^{i}_{\OO_{L_i}}(\log)\to\hat{\Omega}^{i}_{\OO_{L_i}/\OO_{L_{i-1}}}(\log) \to \hat{\Omega}^{i-1}_{\OO_{L_{i-1}}}(\log),\] where $\hat{\Omega}^{i}_{\OO_{L_i}/\OO_{L_{i-1}}}(\log) \to \hat{\Omega}^{i-1}_{\OO_{L_{i-1}}}(\log)$ is the homomorphism that satisfies 			
	\[
	a d\log T_1\wedge\cdots\wedge d\log T_{i} \mapsto c_{L_i/L_{i-1}}(a)d\log T_1\wedge\cdots\wedge d\log T_{i-1}
	\]
	for $a\in\OO_{L_i}$.
	Then define	
	\[
	\Res_{L/K}:\hat{\Omega}^{q-1}_{\OO_L} (\log)\to \OO_{K}
	\]
	as the composition \[\Res_{L/K} = \Res_{L_1/L_0}\circ \cdots\circ\Res_{L_{q-1}/L_{q-2}}.\]
	It induces 
	\[
	\Res_{L/K}:\hat{\Omega}^{q-1}_{\OO_L} (\log)\otimes_{\OO_L}L\to K.
	\]
\end{definition}

\begin{definition} \label{def:residueM}Let $L$ be a finite extension of $M=K\{\{T_1\}\}\cdots\{\{T_{q-1}\}\}$, where $K$ is a local field of mixed characteristic.   Define the residue map 
	\[
	\Res_{L/K}:\hat{\Omega}^{q-1}_{\OO_L} (\log)\otimes_{\OO_L}L\to K
	\]
	by
	\[
	\Res_{L/K}=\Res_{M/K}\circ \Tr_{L/M}.
	\]	
\end{definition}

\begin{remark}
	In Definition \ref{def:residueM}, $\Res_{L/K}$ is expected to be independent of $M$. Independence has been proven when $L$ is a two-dimensional local field (\cite[2.3.3]{morrow2010explicit}), but appears to remain open in the general case. This property shall not be necessary for us.   
\end{remark}

We will now start to obtain some properties of the trace and residue maps that will be necessary for the proof of the main theorem of this section.

\begin{proposition}\label{the:res}
	Let $L$ be a complete discrete valuation field that is a finite extension of $M=K\{\{T_1\}\}\cdots\{\{T_{q-1}\}\}$, where $K$ is a local field of mixed characteristic. Write $e= e(L/K)$.
	Then, for any integer $n$,
	\[
	\Res_{L/K}\left(\m_L^{ne-\delta_\tor(L/K)}\dfrac{\hat{\Omega}^{q-1}_{\OO_L}(\log)}{\hat{\Omega}^{q-1}_{\OO_L}(\log)_\tor}
	\right) = \m_K^n
	\]
	and
	\[
	\Res_{L/K}\left(\m_L^{ne-\delta_\tor(L/K)+1}\dfrac{\hat{\Omega}^{q-1}_{\OO_L}(\log)}{\hat{\Omega}^{q-1}_{\OO_L}(\log)_\tor}\right) = \m_K^{n+1}.
	\]
\end{proposition}
\begin{proof}We shall prove the first equality; the second is obtained in a similar way. 
	
	Observe that 
	$\Res_{L/K}=\Res_{M/K}\circ \Tr_{L/M}$. Further, $\hat{\Omega}^{q-1}_{\OO_M}(\log)$ is generated by $dT_i$ and $d\log\pi_K$, and its torsion part is generated by $d\log\pi_K$. Thus we have an isomorphism $\OO_M \otimes_{\OO_K} \hat{\Omega}^1_{\OO_K}(\log) \simeq \hat{\Omega}^{q-1}_{\OO_M}(\log)_\tor$. We get, by definition,
	\[
	\delta_\tor(L/K)= \delta_\tor(L/M).
	\]
	Then, using Lemma \ref{lemma:trace}, we get
	\begin{align*}
	\Res_{L/K}\left(\m_L^{ne-\delta_\tor(L/K)}\dfrac{\hat{\Omega}^{q-1}_{\OO_L}(\log)}{\hat{\Omega}^{q-1}_{\OO_L}(\log)_\tor}
	\right) &=\\ \Res_{M/K}\left(\Tr_{L/M}\left(\m_L^{ne-\delta_\tor(L/K)}\dfrac{\hat{\Omega}^{q-1}_{\OO_L}(\log)}{\hat{\Omega}^{q-1}_{\OO_L}(\log)_\tor}\right)
	\right)&=\\
	\Res_{M/K}\left(\m_M^n\dfrac{\hat{\Omega}^{q-1}_{\OO_M}(\log)}{\hat{\Omega}^{q-1}_{\OO_M}(\log)_\tor}
	\right)&=\m_K^n.
	\qedhere \end{align*}
\end{proof}

\begin{proposition} \label{cor:2}
	Let $L$, $K$,  and $e$ be as in Proposition  \ref{the:res}.    Then, if $n\in\N$ satisfies  
	\[n\geq \frac{\delta_\tor(L/K)}{e},\]
	we have 
	\[
	\Res_{L/K}\left(\m_L^{ne-\delta_\tor(L/K)}\hat{\Omega}^{q-1}_{\OO_L}(\log)
	\right) = \m_K^n
	\]
	and
	\[
	\Res_{L/K}\left(\m_L^{ne-\delta_\tor(L/K)+1}\hat{\Omega}^{q-1}_{\OO_L}(\log)\right) = \m_K^{n+1}.
	\]
\end{proposition}
\begin{proof}In this case $en-\delta_\tor(L/K)\geq 0$, so this follows from  Proposition \ref{the:res}.
\end{proof}

We will now use the previous properties of residue and trace maps,  the exponential map defined by M. Kurihara (\cite{kurihara1998exponential}), and a modification of higher dimensional class field theory to prove that, when $L$ is a $q$-dimensional local field that is  a finite extension  of  $K\{\{T_1\}\}\cdots\{\{T_{q-1}\}\}$, Main Result \ref{main re mix} holds. This will then be used to prove the general result. We start with the following theorem:

\begin{theorem} \label{key2}
	Let $L$ be a $q$-dimensional local field that is a finite extension of  $M=K\{\{T_1\}\}\cdots\{\{T_{q-1}\}\}$, where $K$ is a local field of mixed characteristic with residue field $k$ of characteristic $p>0$. 
	Write $e= e(L/K)$.  
	Assume that $n\in\N$ satisfies  
	\[n\geq \frac{2e_K}{p-1}+\frac{1}{e}\]
	and let $n'\in\N$ be such that $n'\geq \frac{\delta_\tor(L/K)}{e}$ . Take $\eta\in\OO_K$ such that $v_K(\eta)=n$. 
	
	Then we have a commutative diagram 
	\begin{center}
		\begin{tikzcd}
			\m_L^{en'-\delta_\tor(L/K)}	\hat{\Omega}^{q-1}_{\OO_{L}}(\log) \arrow[r, "\exp_\eta"] \arrow[d, "\Res_{L/K}"] 
			& \hat{K}_q(L) \arrow[d, "\Res_{L/K}"] \\
			\m_K^{n'} \arrow[r, "\exp_\eta"]
			& K^\times
		\end{tikzcd}
	\end{center}
	where the right vertical arrow is the residue homomorphism from $K$-theory defined in \cite{kato1983residue} and the top and bottom horizontal maps are, respectively, the exponential maps $\exp_{\eta,q}$ and $\exp_{\eta,1}$ defined in \cite{kurihara1998exponential}. Further, the left vertical arrow is surjective. 
\end{theorem}
\begin{proof}
	First, observe that the condition
	\[
	n\geq \frac{2e_K}{p-1} + \frac{1}{e}
	\]
	implies
	\[
	en \geq \frac{2e_K e}{p-1}+1=\frac{2e_L}{p-1}+1.
	\] 
	Therefore this condition guarantees the convergence of both the top and the bottom exponential maps (by Theorem 0.1 in \cite{kurihara1998exponential}). Furthermore, the condition 
	\[
	n'\geq   \frac{\delta_\tor(L/K)}{e}
	\] 
	guarantees that we can apply Proposition \ref{cor:2}.

	We need to prove that the diagram  
	\begin{center}
		\begin{tikzcd}
			\m_L^{en'-\delta_\tor(L/K)}	\hat{\Omega}^{q-1}_{\OO_{L}}(\log) \arrow[r,"\exp_\eta"] \arrow[d, "\Tr_{L/M}"] 
			& \hat{K}_q(L) \arrow[d, "N_{L/M}"] 								
			\\
			\m_M^{n'}	\hat{\Omega}^{q-1}_{\OO_{M}}(\log) \arrow[r, "\exp_\eta"] \arrow[d, "\Res_{M/K}"] 
			& \hat{K}_q(M) \arrow[d, "\Res_{M/K}"] \\
			\m_K^{n'} \arrow[r, "\exp_\eta"]
			& K^\times
		\end{tikzcd}
	\end{center}
	commutes. 
	
	By Proposition \ref{cor:2}, the map $\Res_{L/K}:	\hat{\Omega}^{q-1}_{\OO_{L}}(\log)\to \OO_K$ induces a surjection \[\Res_{L/K}:	\m_L^{en'-\delta_\tor(L/K)}	\hat{\Omega}^{q-1}_{\OO_{L}}(\log)\twoheadrightarrow  \m_K^{n'},\] and the map  $\Res_{M/K}:		\hat{\Omega}^{q-1}_{\OO_{M}}(\log)\to  \OO_K$ induces a surjection \[\Res_{M/K}:	\m_M^{n'}	\hat{\Omega}^{q-1}_{\OO_{M}}(\log)\twoheadrightarrow  \m_K^{n'}.\] A similar argument shows that
	$\Tr_{L/M}:	\hat{\Omega}^{q-1}_{\OO_{L}}(\log)\to	\hat{\Omega}^{q-1}_{\OO_{M}}(\log)$ induces 
	a surjection  \[\Tr_{L/M}:	\m_L^{en'-\delta_\tor(L/K)}	\hat{\Omega}^{q-1}_{\OO_{L}}(\log)\twoheadrightarrow	\m_M^{n'}	\hat{\Omega}^{q-1}_{\OO_{M}}(\log).\]

	The commutativity of the top square is shown in
	\cite{kurihara1998exponential}. 
	The commutativity of the bottom square can be checked explicitly as follows. Let $M_0=K, M_1=K\{\{T_1\}\}, \ldots,  M_{q-1}=M= K\{\{T_1\}\}\cdots\{\{T_{q-1}\}\}$. It is enough to show that each one of the squares in the diagram 
	\begin{center}
		\begin{tikzcd}
			\hat{\Omega}^{q-1}_{\OO_{M}}(\log) \arrow[r, "\exp_\eta"] \arrow[d,  "\Res_{M/M_{q-2}}"] 
			& \hat{K}_q(M) \arrow[d, "\Res_{M/M_{q-2}}"] \\
			\hat{\Omega}^{q-2}_{\OO_{M_{q-2}}}(\log) \arrow[r, "\exp_\eta"] \arrow[d,  "\Res_{M_{q-2}/M_{q-3}}"] 
			& \hat{K}_{q-1}(M_{q-2}) \arrow[d, "\Res_{M_{q-2}/M_{q-3}}"]\\
			\vdots \arrow[d, "\Res_{M_{2}/M_{1}}"]& \vdots \arrow[d, "\Res_{M_{2}/M_{1}}"]
			\\
			\hat{\Omega}^{1}_{\OO_{M_{1}}}(\log)\arrow[r, "\exp_\eta"] \arrow[d, "\Res_{M_{1}/K}"] 
			&  \hat{K}_{2}(M_{1})  \arrow[d, "\Res_{M_{1}/K}"]\\
			
			\OO_K \arrow[r, "\exp_\eta"]
			& K^\times
		\end{tikzcd}
	\end{center}
	commutes. 
	
	Let $a\in\OO_{M_i}$ and write \[a= \sum\limits_{k<0} a_k T_i^k + a_0 + \sum\limits_{k>0} a_k T_i^k,\]
	where $a_k \in \OO_{M_{i-1}}$ for every $k\in\Z$. Put $a_{-}=\sum\limits_{k<0} a_k T_i^k$ and $a_+=\sum\limits_{k>0} a_k T_i^k$.
	Observe first that, since
	\begin{center}
		\begin{tikzcd}
			\hat{K}_{i}(M_{i-1}) \arrow[rr, "\{ \, {,}  T_i\}"] && \hat{K}_{i+1}(M_{i}) \arrow[rr, "\Res_{M_i/M_{i-1}}"] && 	\hat{K}_{i}(M_{i-1})
		\end{tikzcd}
	\end{center} 
	is the identity map (\cite[Theorem 1]{kato1983residue}), we get
	\begin{align*} \Res_{M_i/M_{i-1}}\circ\exp_\eta\left(a_0 \frac{d T_1}{T_1}\wedge\cdots\wedge \frac{d T_{i}}{T_{i}}\right)&= \\ \Res_{M_i/M_{i-1}}\{\exp(\eta a_0),T_1,\ldots,T_{i}\}=\{\exp(\eta a_0),T_1,\ldots,T_{i-1}\}&=\\\exp_\eta\circ\Res_{M_i/M_{i-1}}\left(a_0 \frac{d T_1}{T_1}\wedge\cdots\wedge \frac{d T_{i}}{T_{i}}\right).
	\end{align*}
	Further, the same theorem gives 
	\begin{align*}
	\Res_{M_i/M_{i-1}}\circ\exp_\eta\left(a_+ \frac{d T_1}{T_1}\wedge\cdots\wedge \frac{d T_{i}}{T_{i}}\right)
	&= \\
	\exp_\eta\circ\Res_{M_i/M_{i-1}}\left(a_+ \frac{d T_1}{T_1}\wedge\cdots\wedge \frac{d T_{i}}{T_{i}}\right)&=0.
	\end{align*}
	We will now show that we also have
	\begin{align*}
	\Res_{M_i/M_{i-1}}\circ\exp_\eta\left(a_- \frac{d T_1}{T_1}\wedge\cdots\wedge \frac{d T_{i}}{T_{i}}\right)
	&=\\
	\exp_\eta\circ\Res_{M_i/M_{i-1}}\left(a_- \frac{d T_1}{T_1}\wedge\cdots\wedge \frac{d T_{i}}{T_{i}}\right)&=0.
	\end{align*}
	From Theorem 1 in \cite{kato1983residue}, 
	we have,  for $k\in\Z_{<0}$ and $m\in \N$,
	\[	\Res_{M_i/M_{i-1}}\left\{1 + \eta a_kT_i^k+\cdots + \frac{(\eta a_k)^mT_i^{mk}}{m!},T_1,\ldots,T_{i}\right\}= 0.\] 
	Since $v_{M_{i-1}}((\eta a_k)^m/m!)\to \infty$ and the residue map is continuous, we have 
	\[ \label{equation}\tag{$\ast$}
	\Res_{M_i/M_{i-1}}\left\{\exp(\eta a_k T_i^k),T_1,\ldots,T_{i}\right\}= 0.
	\]	 
	Given $k\in \Z_{<0}$, write $s_{k}= \sum\limits_{k\leq k' <0} a_{k'}T_i^{k'}$. 
	 From  (\ref{equation}) we have  that  	\[\Res_{M_i/M_{i-1}}\left\{\exp(\eta s_k),T_1,\ldots,T_{i}\right\}= 0.\]
	By continuity 
	and $s_k\to a_-$, we get 
	\[
	\Res_{M_i/M_{i-1}}\left\{\exp(\eta a_-),T_1,\ldots,T_{i}\right\}= 0.
	\]	 	
	Hence we conclude that 
	\begin{align*}
	\Res_{M_i/M_{i-1}}\circ\exp_\eta\left(a \frac{d T_1}{T_1}\wedge\cdots\wedge \frac{d T_{i}}{T_{i}}\right)&=\\
	\exp_\eta\circ\Res_{M_i/M_{i-1}}\left(a \frac{d T_1}{T_1}\wedge\cdots\wedge \frac{d T_{i}}{T_{i}}\right)&.
	\end{align*}
	A similar argument shows that  
	\begin{align*}
	\Res_{M_i/M_{i-1}}\circ\exp_\eta\left(a \frac{d T_1}{T_1}\wedge\cdots\wedge \frac{d T_{i-1}}{T_{i-1}}\wedge \frac{d \pi_K}{\pi_K}\right)&=\\
	\exp_\eta\circ\Res_{M_i/M_{i-1}}\left(a \frac{d T_1}{T_1}\wedge\cdots\wedge \frac{d T_{i-1}}{T_{i-1}}\wedge \frac{d\pi_K}{\pi_K}\right)&=0,
	\end{align*}
	so we conclude that each square in the diagram is commutative. 
\end{proof}

We have now developed all the necessary tools in order to prove Proposition  \ref{most}, which states that Main Result \ref{main re mix} holds when 	  $L$ is a $q$-dimensional local field that is a  finite extension  of  $K\{\{T_1\}\}\cdots\{\{T_{q-1}\}\}$. We will then use
Proposition \ref{most} to prove Theorem \ref{main mixed}, which gives  Main Result \ref{main re mix} in full generality.

\begin{proposition} \label{most}
	Let $L$ be a $q$-dimensional local field that is a finite extension of   $M=K\{\{T_1\}\}\cdots\{\{T_{q-1}\}\}$, where $K$ is a local field of mixed characteristic with residue field $k$ of characteristic $p>0$.
	Assume that  $\chi\in H^1(K)$ is   such that \[\sw \chi \geq
	\dfrac{2e_K}{p-1}+\dfrac{1}{e(L/K)}+\left\lceil\dfrac{\delta_\tor(L/K)}{e(L/K)}\right\rceil.\] Denote by $\chi_L$ its image in $H^1(L)$. Then 
	\[\sw \chi_L = e(L/K) \sw \chi - \delta_\tor(L/K).\]
\end{proposition}
\begin{proof}
	Using the same argument as in \cite[(7.6)]{kato1989swan}, we can assume  $H_p^1(k)\neq 0$. Let $L = L_q, l=L_{q-1}, \ldots, L_1, L_0$ be the chain of residue fields of the $q$-dimensional local field $L$. Since there are isomorphisms (\cite[Theorem 3]{kato1982galois})
	\[
	H^{q+1}(L)\{p\}\simeq H^{q}(L_{q-1})\{p\}\simeq H^{q-1}(L_{q-2})\{p\}\simeq\cdots \simeq H^1(L_0)\{p\}
	\] 
	and 
	\[
	H^2(K)\{p\}\simeq H^1(k)\{p\},
	\]
	we have a commutative diagram 
	\begin{center}
		\begin{tikzcd}[column sep=small]
			H^1(L) &\times & \hat{K}_q(L) \arrow[dd, "\Res_{L/K}"] \ar{rrr}{\{ \; , \; \}_{L}} & & & H^{q+1}(L)\{p\} \arrow[rrr, "\simeq"] & & & H^1(L_0) \{p\} \arrow[dd] \\
			\\ H^1(K) \arrow[uu] & \times &
			K^\times \ar{rrr}{\{ \; , \; \}_K} & & & H^{2}(K)\{p\} \arrow[rrr, "\simeq"] & & & H^1(k) \{p\}
		\end{tikzcd}
	\end{center}
	Here, the pairing $H^1(L)\times \hat{K}_q(L)\to H^{q+1}(L)\{p \}$ is the one constructed  in \cite{kato1989swan}. 
	Denote the composition  $H^1(L)\times \hat{K}_q(L)\to H^{q+1}(L)\{p \}\to H^{1}(k)\{p \}$ by $\{ \; , \; \}_k$. Similarly, denote the composition  $H^1(L)\times \hat{K}_q(L)\to H^{q+1}(L)\{p \}\to H^{q}(l)\{p \}$ by $\{ \; , \; \}_l$. Since the last arrow is an isomorphism,  $\{A, B\}_L=0$ if and only if $\{A, B\}_l=0$, where $A\in H^1(L)$ and $B\in \hat{K}_q(L)$. 
	
	Observe that $H^1_p(L_0)\neq 0$. Indeed,  $H^1_p(L_0)\simeq L_0/(x^p-x, x\in L_0)$ and
	$H^1_p(k)\simeq k/(x^p-x, x\in k)$, so $H^1_p(L_0) \twoheadrightarrow H^1_p(k)$ follows from the compatibility between the corestriction map and the trace map. Since $H^1_p(k)\neq 0$, we also have $H^1_p(L_0)\neq 0$. 
	
	 From \cite[Proposition 6.5]{kato1989swan}, we have that 
	\[
	\sw \chi_L =m\geq 1
	\]				
	if and only if 
	\[
	\{\chi_L, U^{m+1}\hat{K}_q(L)\}_L=0
	\]
	but
	\[
	\{\chi_L, U^{m}\hat{K}_q(L)\}_L\neq 0.
	\]

		To simplify notation, put $e=e(L/K)$, $n'=\left\lceil\frac{\delta_\tor(L/K)}{e}\right\rceil$ and $n=\sw \chi - n'$.  Pick $\eta \in \OO_K$ such that $v_K(\eta)=n$. From Lemma \ref{lemma:exp}, the commutative   diagram
	\begin{center}
		\begin{tikzcd}
			\m_L^{en'-\delta_\tor(L/K)}	\hat{\Omega}^{q-1}_{\OO_{L}}(\log) \arrow[r, "\exp_\eta"] \arrow[d, "\Res_{L/K}"] 
			& \hat{K}_q(L) \arrow[d, "\Res_{L/K}"]    \\
			\m_K^{n'} \arrow[r, "\exp_\eta"]
			& K^\times 
		\end{tikzcd}
	\end{center}
	given  by Theorem \ref{key2}, and the surjectivity of the left vertical arrow, we have that 
	\begin{align*}
	\{\chi_L, U^{e\sw\chi-\delta_\tor(L/K)+1}\hat{K}_q(L)\}_k&=\\ \{\chi_L, U^{en'-\delta_\tor(L/K)+en+1}\hat{K}_q(L)\}_k&=\{\chi, U^{\sw \chi+1}_K\}_k=0
	\end{align*}
	but
	\begin{align*}
	\{\chi_L, U^{e\sw \chi-\delta_\tor(L/K)}\hat{K}_q(L)\}_k&=\\
	\{\chi_L, U^{en'-\delta_\tor(L/K)+en}\hat{K}_q(L)\}_k&=	\{\chi, U^{\sw \chi}_K\}_k\neq 0.
	\end{align*}
	
	This clearly yields $\{\chi_L, U^{e\sw\chi-\delta_\tor(L/K)}\hat{K}_q(L)\}_L\neq 0$, so $\sw \chi_L \geq e \sw \chi - \delta_\tor(L/K)$. It remains to show that $\sw \chi_L \leq e \sw \chi - \delta_\tor(L/K)$. 

	Assume that $s=\sw \chi_L > e \sw \chi -\delta_\tor(L/K)$. 
	The key point is to show that
	 \[\{\chi_L, U^{s}\hat{K}_q(L)\}_l\supset H_p^{q}(l).\]
	 Indeed, if $\{\chi_L, U^{s}\hat{K}_q(L)\}_l\supset H_p^{q}(l)$, then,  from the isomorphisms
	 \[
	  H_p^{q}(l) \simeq \cdots \simeq H_p^1(L_0)
	 \]
	obtained in \cite{kato1982galois} and the surjectivity of $H_p^1(L_0)\twoheadrightarrow H_p^1(k)$, we get that 
	\[\{\chi_L, U^{s}\hat{K}_q(L)\}_k\neq0.\]
	This is a contradiction because
	\[
		\{\chi_L, U^{s}\hat{K}_q(L)\}_k \subset \{\chi_L, U^{ e \sw \chi -\delta_\tor(L/K) +1}\hat{K}_q(L)\}_k = \{\chi, U_K^{\sw \chi +1}\}_k = 0.
	\]
	
	We will now show that $\{\chi_L, U^{s}\hat{K}_q(L)\}_l\supset H_p^{q}(l)$. Since $l$ is of characteristic $p>0$, there is an isomorphism 
	\[
		H_p^q(l) \simeq \Coker\left(F-1:\Omega_l^{q-1}\longrightarrow
		 \Omega_l^{q-1}/d\Omega_l^{q-2}\right).
	\]
	Denote by $\delta_1(\omega)$ the class of $\omega \in \Omega_l^{q-1}$ in $H_p^q(l)$. Let  $[\pi_L^s]^{-1}\left(\alpha + \beta \frac{d[\pi_L]}{[\pi_L]}\right)$ be Kato's refined Swan conductor (\cite[Definition 5.3]{kato1989swan}) of $\chi_L$, where $\alpha \in \Omega^1_l$ and $\beta \in l$, and $(\alpha,\beta)\neq (0,0)$.
	
	If $\beta \neq 0$,  take $u_1, \ldots, u_{q-1} \in l^\times$ and $a\in l$ such that 
	\[\delta_1\left(a\beta \frac{du_1}{u_1} \wedge \cdots \wedge \frac{d u_{q-1}}{u_{q-1}}\right)\neq 0.\] 
	Let $\tilde{a}\in \OO_L$, $\tilde{u_i}$ be lifts of $a$, $u_i$ to $\OO_L$.  Then
	\[\{\chi_L, 1+\tilde{a} \pi_L^s, \tilde{u}_1, \ldots, \tilde{u}_{q-1}\}_l=\delta_1\left(a\beta \frac{du_1}{u_1} \wedge \cdots \wedge \frac{d u_{q-1}}{u_{q-1}}\right)\neq 0.\]
	Since  $\Omega_l^{q-1}$ is a one-dimensional vector space over $l$,  we know that the element ${\beta \frac{du_1}{u_1} \wedge \cdots \wedge \frac{d u_{q-1}}{u_{q-1}}}$ is a generator for    $\Omega_l^{q-1}$ over $l$. Then
	\begin{align*}
		H_p^{q}(l)&=\left\{ \delta_1\left(b\beta \frac{du_1}{u_1} \wedge \cdots \wedge \frac{d u_{q-1}}{u_{q-1}} \right): b\in l \right\}\\&= \left\{\{\chi_L, 1+\tilde{b} \pi_L^s, \tilde{u}_1, \ldots, \tilde{u}_{q-1}\}_l:\tilde{b}\in \OO_L \right\}\subset\{\chi_L, U^{s}\hat{K}_q(L)\}_l .
	\end{align*} 
	Similarly, if $\beta = 0$ and $\alpha \neq 0$, take  $u_1, \ldots, u_{q-2} \in l^\times$ and $a\in l$ such that 
	\[\delta_1\left(a\alpha \wedge\frac{du_1}{u_1} \wedge \cdots \wedge \frac{d u_{q-2}}{u_{q-2}}\right)\neq 0.\] 
	We have that 
	\[\{\chi_L, 1+\tilde{a} \pi_L^s, \tilde{u}_1, \ldots, \tilde{u}_{q-2}, \pi_L\}_l=\delta_1\left(a\alpha \wedge \frac{du_1}{u_1} \wedge \cdots \wedge \frac{d u_{q-1}}{u_{q-2}}\right)\neq 0,\]
	and $\alpha \wedge \frac{du_1}{u_1} \wedge \cdots \wedge \frac{d u_{q-1}}{u_{q-2}}$ is a generator for $\Omega_l^{q-1}$ over $l$. Then, using the same reasoning as before, we get
	$H_p^{q}(l)\subset \{\chi_L, U^{s}\hat{K}_q(L)\}_l$.
\end{proof}

\begin{theorem} \label{main mixed}
	Let $L/K$ be an 
	extension of complete discrete valuation fields of mixed characteristic. Assume that $K$ has perfect residue field of characteristic $p>0$. 
	
	Denote by $e(L/K)$ the ramification index of $L/K$.
	Assume that  $\chi\in H^1(K)$ is such that \[\sw \chi \geq
	\dfrac{2e_K}{p-1}+\dfrac{1}{e(L/K)}+\left\lceil\dfrac{\delta_\tor(L/K)}{e(L/K)}\right\rceil.\] Denote by $\chi_L$ its image in $H^1(L)$. Then 
	\[\sw \chi_L = e(L/K) \sw \chi - \delta_\tor(L/K).\]
\end{theorem}
\begin{proof}
	Following the same argument as \cite[\S 10]{kato1989swan}, we can assume that  the residue field $l$ of $L$ is finitely generated over the residue field $k$ of $K$.
	Since we have proven Proposition \ref{most}, it is enough to show that this  case can be reduced to that  of a $q$-dimensional local field that is a finite extension of  $K\{\{T_1\}\}\cdots\{\{T_{q-1}\}\}$.
	
	Since $l$ is finitely generated over $k$, there are $T_1, \ldots, T_{q-1} \in l$ such that $l$ is a finite, separable extension of $k(T_1,\ldots, T_{q-1})$. Since there is an embedding  $k(T_1,\ldots, T_{q-1})\hookrightarrow k((T_1))\cdots ((T_{q-1}))$,
	there is also an embedding $l\hookrightarrow E$ of $l$ into a finite, separable extension $E$ of $k((T_1))\cdots ((T_{q-1}))$. Since $\{T_1, \ldots, T_{q-1}\}$ is a $p$-basis for both $l$ and $E$, there is a complete, discrete valuation field $L(E)$ that is an extension of $L$ satisfying $\OO_L\subset \OO_{L(E)}$, $\m_L\subset \m_{L(E)}$,  $\pi_L$ is still prime in $L(E)$, and the residue field of $L(E)$ is isomorphic to $E$ over $l$.
	
	$L(E)$ is a finite extension of $K\{\{T_1\}\}\cdots\{\{T_{q-1}\}\}$. Since  $e(L(E)/L)=1$, we get $e(L(E)/K)=e(L/K)$. Further, since $E$ and $l$ have the same $p$-basis and $\pi_L$ is a prime for both $L$ and $L(E)$, the map $\OO_{L(E)}\otimes_{\OO_L} \hat{\Omega}^1_{\OO_{L}}(\log)\to \hat{\Omega}^1_{\OO_{L(E)}}(\log)$ sends generators to generators satisfying the same relations, so it is an isomorphism. In particular,  $\OO_{L(E)}\otimes_{\OO_L} \hat{\Omega}^1_{\OO_{L}}(\log)_\tor\simeq \hat{\Omega}^1_{\OO_{L(E)}}(\log)_\tor$. Therefore, by definition, $\delta_\tor(L(E)/K)=\delta_\tor(L/K)$. From \cite[Lemma 6.2]{kato1989swan}, since  $\OO_L\subset \OO_{L(E)}$, $\m_{L(E)}=\OO_{L(E)} \m_{L}$, and the extension of residue fields is separable, we have  $\sw \chi_{L(E)}=\sw \chi_{L}$. Thus it is sufficient to prove that
	\[\sw \chi_{L(E)} = e(L/K) \sw \chi - \delta_\tor(L(E)/K),\]
	which follows from Proposition \ref{most}. 
\end{proof}

\section{A generalized $\psi$-function} \label{sec5}

Through  this section, let $L/K$ be an 
extension of  complete discrete valuation fields such that the
residue field of $K$ is perfect and 
of characteristic $p>0$. We define   
generalizations of the classical $\psi$-function for this case. More precisely, we will define functions $\psi_{L/K}^{\mathrm{AS}}:\R_{\geq 0}\to \R_{\geq 0}$ and $\psi_{L/K}^{\mathrm{ab}}:\R_{\geq 0}\to\R_{\geq 0}$ and show that, in the classical case of $L/K$ finite, they both coincide with the classical $\psi_{L/K}:\R_{\geq 0}\to\R_{\geq 0}$ (see Theorem \ref{thm:psi}). The superscripts AS and ab refer, respectively, to  Abbes-Saito and abelian. In the definition of  $\psi_{L/K}^{\mathrm{AS}}$ we use the Abbes-Saito upper ramification filtrations of absolute Galois groups, while  in the definition of  $\psi_{L/K}^{\mathrm{ab}}$ we use Kato's ramification filtration of $H^1(L)$. 

We also define functions $\varphi_{L/K}^{\mathrm{AS}}:\R_{\geq 0}\to \R_{\geq 0}$  and $\varphi_{L/K}^{\mathrm{ab}}:\R_{\geq 0}\to \R_{\geq 0}$ and show that, when
$\varphi_{L/K}^{\mathrm{AS}}$ and $\varphi_{L/K}^{\mathrm{ab}}$  are injective, $\psi_{L/K}^{\mathrm{AS}}$ and $\psi_{L/K}^{\mathrm{ab}}$ are their respective left inverses (and vice-versa).
 
Assume first that the residue field $k$ of $K$ is algebraically closed. 
For $t \in \Z_{(p)}$, $t\geq0$, define $\psi^{\mathrm{ab}}_{L/K}(t)\in \R_{\geq 0}$ as 
\begin{align*}
	&\psi^{\mathrm{ab}}_{L/K}(t)= \\ &\inf\left\lbrace s\in \Z_{(p)} \, \middle| \, \begin{array}{l@{}l@{}}
		\operatorname{Im}(F_{e(K'/K)t}H^1(K') \to H^1(LK')) \subset F_{e(LK'/L)s}H^1(LK')   \\  \text{for all finite, tame extensions $K'/K$ of complete discrete} \\ \text{valuation fields such that }  e(LK'/L)s, e(K'/K)t\in \Z  \end{array}  \right\rbrace,
\end{align*}
and then extend $\psi^{\mathrm{ab}}_{L/K}$ to $\R_{\geq 0}$ by putting 
\[
\psi^{\mathrm{ab}}_{L/K}(t)=\sup\{\psi^{\mathrm{ab}}_{L/K}(s):s\leq t, s \in \Z_{(p)} \}.
\]

Similarly, for $t \in \Z_{(p)}$, $t\geq0$, define $\varphi^{\mathrm{ab}}_{L/K}(t)\in \R_{\geq 0}$ as 
\begin{align*}
	&\varphi^{\mathrm{ab}}_{L/K}(t)= \\& \sup\left\lbrace s\in \Z_{(p)} \, \middle| \, \begin{array}{l@{}l@{}}
		\operatorname{Im}(F_{e(K'/K)s}H^1(K') \to H^1(LK')) \subset F_{e(LK'/L)t}H^1(LK') \\ \text{for all finite, tame extensions $K'/K$ of complete discrete} \\ \text{valuation fields such that }  e(LK'/L)t, e(K'/K)s\in \Z   \end{array}  \right\rbrace,
\end{align*}
and then extend $\varphi^{\mathrm{ab}}_{L/K}$ to $\R_{\geq 0}$  by putting 
\[
\varphi^{\mathrm{ab}}_{L/K}(t)=\sup\{\varphi^{\mathrm{ab}}_{L/K}(s):s\leq t, s \in \Z_{(p)} \}.
\]

Let  $G_{K, \, \log}^{t+}$ denote the Abbes-Saito logarithmic upper ramification filtration defined in \cite{abbessaito}. We now define  $\psi_{L/K}^{\mathrm{AS}}$ and  $\varphi_{L/K}^{\mathrm{AS}}$ by putting,
for $t \in \R_{\geq 0}$,  
\[\psi^{\mathrm{AS}}_{L/K}(t)=  \inf\left\lbrace s\in \R :  
\operatorname{Im}(G_{L, \, \log}^{s+}\to G_K)\subset G_{K, \, \log}^{t+}     \right\rbrace
\]
and  
\[\varphi^{\mathrm{AS}}_{L/K}(t)=  \sup\left\lbrace s\in \R :
\operatorname{Im}(G_{L, \, \log}^{t+}\to G_K)\subset G_{K, \, \log}^{s+}     \right\rbrace.
\]

When $k$ is not necessarily algebraically closed, we define  $\psi^{\mathrm{ab}}_{L/K}$, 
$\varphi^{\mathrm{ab}}_{L/K}$, $\psi^{\mathrm{AS}}_{L/K}$ and $\varphi^{\mathrm{ab}}_{L/K}$ as follows. Let $\tilde{K}= \widehat{K_{\mathrm{ur}}}$ and $\tilde{L}=\widehat{L K_{ur}}$. 
Then define  $\psi^{\mathrm{ab}}_{L/K}= \psi^{\mathrm{ab}}_{\tilde{L}/\tilde{K}}$, $\varphi^{\mathrm{ab}}_{L/K}= \varphi^{\mathrm{ab}}_{\tilde{L}/\tilde{K}}$, $\psi^{\mathrm{AS}}_{L/K}= \psi^{\mathrm{AS}}_{\tilde{L}/\tilde{K}}$  and 
$\varphi^{\mathrm{AS}}_{L/K}= \varphi^{\mathrm{AS}}_{\tilde{L}/\tilde{K}}$.

The above defined functions have properties similar to those of their classical counterparts. We will now prove some of these properties. 

\begin{proposition}
	If $\varphi^{\mathrm{ab}}_{L/K}(t)$ is injective, then  $\psi^{\mathrm{ab}}_{L/K}(t)$  is its left inverse. Similarly, if $\psi^{\mathrm{ab}}_{L/K}(t)$ is injective, then  $\varphi^{\mathrm{ab}}_{L/K}(t)$ is its left inverse. 
\end{proposition}
\begin{proof}
	From the definitions of  $\varphi^{\mathrm{ab}}_{L/K}(t)$ and $\psi^{\mathrm{ab}}_{L/K}(t)$, we can assume that $k$ is algebraically closed.
	We shall prove that if $\varphi^{\mathrm{ab}}_{L/K}(t)$ is injective, then  $\psi^{\mathrm{ab}}_{L/K}(t)$  is its left inverse. 	The other statement is proved in an analogous way. 
	
	It is enough to show that, for $t\in \Z_{(p)}$, $t\geq 0$, we have $t=\psi^{\mathrm{ab}}_{L/K}(\varphi^{\mathrm{ab}}_{L/K}(t))$.
	If $s\in \Z_{(p)}$ is smaller or equal to  $\varphi^{\mathrm{ab}}_{L/K}(t)$, then 
	\[
	\operatorname{Im}(F_{e(K'/K)s}H^1(K') \to H^1(LK')) \subset F_{e(LK'/L)t}H^1(LK')
	\]
	for all finite Galois  extensions $K'/K$ of complete discrete valuation fields such that  $K'/K$ is tame and  $e(LK'/L)t, e(K'/K)s\in \Z$. Then
	$t\geq\psi^{\mathrm{ab}}_{L/K}(s)$, so  $t\geq\psi^{\mathrm{ab}}_{L/K}(\varphi^{\mathrm{ab}}_{L/K}(t))$.
	
	Assume that we have $t>\psi^{\mathrm{ab}}_{L/K}(\varphi^{\mathrm{ab}}_{L/K}(t))$. Take  $\tilde{t}\in \Z_{(p)}$ that satisfies 
	$\psi^{\mathrm{ab}}_{L/K}(\varphi^{\mathrm{ab}}_{L/K}(t))<\tilde{t}<t$. Let $K'/K$ be any finite Galois extension of complete discrete valuation fields that is tame and such that $e(LK'/L)\tilde{t}\in \Z$. 
	Since $\tilde{t}>\psi^{\mathrm{ab}}_{L/K}(\varphi_{L/K}^{\mathrm{ab}}(t))$, 
	\[
	\operatorname{Im}(F_{e(K'/K)s}H^1(K') \to H^1(LK')) \subset F_{e(LK'/L)\tilde{t}}H^1(LK')
	\]
	for every $s\leq\varphi_{L/K}^{\mathrm{ab}}(t)$ in $\Z_{(p)}$ such that $e(K'/K)s\in \Z$. Then
	\[
	\varphi_{L/K}^{\mathrm{ab}}(\tilde{t})\geq \varphi_{L/K}^{\mathrm{ab}}(t).
	\]
	Since $\varphi_{L/K}^{\mathrm{ab}}$ is clearly increasing and $t>\tilde{t}$, we get 
	\[
	\varphi_{L/K}^{\mathrm{ab}}(\tilde{t})= \varphi_{L/K}^{\mathrm{ab}}(t),
	\]
	which contradicts the injectivity assumption. Therefore 
	\[
	t=\psi^{\mathrm{ab}}_{L/K}(\varphi^{\mathrm{ab}}_{L/K}(t))
	\]
	for every $t\geq 0$ and we conclude that  $\psi^{\mathrm{ab}}_{L/K}(t)$  is the left inverse of 
	$\varphi^{\mathrm{ab}}_{L/K}(t)$.	
\end{proof}

The analogous result for $\psi_{L/K}^{\mathrm{AS}}$ and $\varphi_{L/K}^{\mathrm{AS}}$ is also true:
\begin{proposition}
	If $\varphi^{\mathrm{AS}}_{L/K}(t)$ is injective, then  $\psi^{\mathrm{AS}}_{L/K}(t)$  is its left inverse. Similarly, if $\psi^{\mathrm{AS}}_{L/K}(t)$ is injective, then  $\varphi^{\mathrm{AS}}_{L/K}(t)$ is its left inverse. 
\end{proposition}
\begin{proof}
	From the definitions of  $\varphi^{\mathrm{AS}}_{L/K}(t)$ and $\psi^{\mathrm{AS}}_{L/K}(t)$, we can assume that $k$ is algebraically closed.
	We shall prove that if $\varphi^{\mathrm{AS}}_{L/K}(t)$ is injective, then  $\psi^{\mathrm{AS}}_{L/K}(t)$  is its left inverse. 	The other statement is proved in an analogous way.
	
	If $s\in \R$ is less than or equal to  $\varphi^{\mathrm{AS}}_{L/K}(t)$, then 
	\[ 
	\operatorname{Im}(G_{L, \, \log}^{t+}\to G_K)\subset G_{K, \, \log}^{s+}.     
	\]
	Hence
	$t\geq\psi^{\mathrm{AS}}_{L/K}(s)\geq\psi^{\mathrm{AS}}_{L/K}(\varphi^{\mathrm{AS}}_{L/K}(t))$.
	
	Assume that we have $t>\psi^{\mathrm{AS}}_{L/K}(\varphi^{\mathrm{AS}}_{L/K}(t))$. Take  $\tilde{t}\in \R$ such that 
	\[\psi^{\mathrm{AS}}_{L/K}(\varphi^{\mathrm{AS}}_{L/K}(t))<\tilde{t}<t.\] Then 
	\[\operatorname{Im}(G_{L, \, \log}^{\tilde{t}+}\to G_K)
	\subset G_{K, \, \log}^{s+} 
	\]
	for every $s\leq\varphi_{L/K}^{\mathrm{AS}}(t)$. Thus 
	\[
	\varphi_{L/K}^{\mathrm{AS}}(\tilde{t})\geq \varphi_{L/K}^{\mathrm{AS}}(t).
	\]
	Since $\varphi_{L/K}^{\mathrm{AS}}$ is clearly increasing and $t>\tilde{t}$, we get 
	\[
	\varphi_{L/K}^{\mathrm{AS}}(\tilde{t})= \varphi_{L/K}^{\mathrm{AS}}(t),
	\]
	which contradicts the injectivity assumption. Therefore 
	\[
	t=\psi^{\mathrm{AS}}_{L/K}(\varphi^{\mathrm{AS}}_{L/K}(t))
	\]
	for every $t\geq 0$ and we conclude that  $\psi^{\mathrm{AS}}_{L/K}(t)$  is the left inverse of 
	$\varphi^{\mathrm{AS}}_{L/K}(t)$. 
\end{proof}

These functions satisfy  formulas similar to those satisfied by the classical $\varphi$ and $\psi$-functions, as we can see from the following lemma. 
\begin{lemma} \label{lemma:psitame}
	Let $K'$ be a finite Galois extension of $K$ that is tamely ramified and $L'=LK'$. Then 
	\begin{eqnarray*}
		\varphi^{\mathrm{ab}}_{L'/K'}(e(L'/L)t)= e(K'/K) \varphi^{\mathrm{ab}}_{L/K}(t), \\
		\psi^{\mathrm{ab}}_{L'/K'}(e(K'/K)t)= e(L'/L) \psi^{\mathrm{ab}}_{L/K}(t), \\
		\varphi^{\mathrm{AS}}_{L'/K'}(e(L'/L)t)= e(K'/K) \varphi^{\mathrm{AS}}_{L/K}(t), \\
		\psi^{\mathrm{AS}}_{L'/K'}(e(K'/K)t)= e(L'/L) \psi^{\mathrm{AS}}_{L/K}(t).
	\end{eqnarray*}
\end{lemma}
\begin{proof}
	Follows from the definitions. For example, 
	\begin{align*}
	\varphi^{\mathrm{AS}}_{L'/K'}(e(L'/L)t) &=  \sup \left\lbrace s\in \R :
	\operatorname{Im}(G_{L, \, \log}^{e(L'/L)t+}\to G_K')\subset G_{K', \, \log}^{s+}     \right\rbrace\\
	&=  \sup \left\lbrace s\in \R :
	\operatorname{Im}(G_{L, \, \log}^{t+}\to G_K)\subset G_{K, \, \log}^{\frac{s}{e(K'/K)}+}     \right\rbrace\\
	&= e(K'/K) \sup \left\lbrace s\in \R :
	\operatorname{Im}(G_{L, \, \log}^{t+}\to G_K)\subset G_{K, \, \log}^{s+}     \right\rbrace\\
	&= e(K'/K) \varphi^{\mathrm{AS}}_{L/K}(t).\qedhere
	\end{align*}
\end{proof}

We relate this section with the rest of our paper. The main results that we proved in the previous sections are, in reality, results about $\psi^{\mathrm{ab}}_{L/K}$. 
More precisely, we have the following theorem: 
\begin{theorem} \label{main thm psi}
	Let $L/K$ be an 
	extension of complete discrete valuation fields. Assume that $K$ has perfect residue field of characteristic $p>0$. Let $t\in\R_{\geq 0}$ be such that
	\[\begin{cases}
	t\geq \dfrac{2e_K}{p-1}+\dfrac{1}{e(L/K)}+\left\lceil\dfrac{\delta_\tor(L/K)}{e(L/K)}\right\rceil & \text{if $K$ is of characteristic $0$,} \\[.5cm] t> \dfrac{p}{p-1} \dfrac{\delta_\tor(L/K)}{e(L/K)} & \text{if $K$ is of characteristic $p$.}
	\end{cases}\]
	Then 
	\[\psi^{\mathrm{ab}}_{L/K}(t)= e(L/K)t-\delta_{\tor}(L/K). \]
\end{theorem} 
\begin{proof}
	Write 
	\[T(L/K)= \dfrac{2e_K}{p-1}+\dfrac{1}{e(L/K)}+\left\lceil\dfrac{\delta_\tor(L/K)}{e(L/K)}\right\rceil \]
	when $K$ is of characteristic $0$, and 
	\[T(L/K)=\dfrac{p}{p-1} \dfrac{\delta_\tor(L/K)}{e(L/K)} \]
	when $K$ is of characteristic $p$. 
	Let $t\in\R_{\geq 0}$ be such that $t\geq T(L/K)$ if $K$ is of characteristic $0$ and $t>T(L/K)$ if $K$ is of characteristic $p$.
	
	If $t\in \Z$, it follows from Theorems \ref{mainpositive} and \ref{main mixed} that
	\[\psi^{\mathrm{ab}}_{L/K}(t)= e(L/K)t-\delta_{\tor}(L/K). \] 
	If $t\in\Z_{(p)}$, take a finite Galois extension $K'/K$ that is tamely ramified and such that  $e(K'/K)t\in \Z$. Observe that, if $K$ is of characteristic $0$,
	\begin{align*}
	e(K'/K)T(L/K) &= \dfrac{2e_{K'}}{p-1}+\dfrac{e(L'/L)}{e(L'/K')}+ e(K'/K)\left\lceil\dfrac{\delta_\tor(L/K)}{e(L/K)}\right\rceil \\
	& \geq  \dfrac{2e_{K'}}{p-1}+\dfrac{e(L'/L)}{e(L'/K')}+ \left\lceil\dfrac{e(L'/L)\delta_\tor(L/K)}{e(L'/K')}\right\rceil \\
	& \geq  \dfrac{2e_{K'}}{p-1}+\dfrac{1}{e(L'/K')}+ \left\lceil\dfrac{\delta_\tor(L'/K')}{e(L'/K')}\right\rceil = T(L'/K').
	\end{align*} 
	Similarly, if $K$ is of characteristic $p$, 
	\begin{align*}
	e(K'/K)T(L/K) &= \dfrac{p}{p-1} \dfrac{e(K'/K)\delta_\tor(L/K)}{e(L/K)}  \\
	& = \dfrac{p}{p-1} \dfrac{e(L'/L)\delta_\tor(L/K)}{e(L'/K')}  \\
	& = \dfrac{p}{p-1} \dfrac{\delta_\tor(L'/K')}{e(L'/K')} = T(L'/K').
	\end{align*} 
	Then we have $e(K'/K)t\geq T(L'/K')$ if $K$ is of characteristic $0$ and $e(K'/K)t > T(L'/K')$ if $K$ is of characteristic $p$. It follows that 
	\begin{align*}
	\psi^{\mathrm{ab}}_{L'/K'}(e(K'/K)t) &= e(L'/K')e(K'/K)t-\delta_{\tor}(L'/K') \\
	&= e(L'/L)e(L/K) t - e(L'/L)\delta_\tor(L/K). 
	\end{align*} 
	From Lemma \ref{lemma:psitame}, we conclude that 
	\[\psi^{\mathrm{ab}}_{L/K}(t)= \dfrac{\psi^{\mathrm{ab}}_{L'/K'}(e(K'/K)t)}{e(L'/L)}= e(L/K)t-\delta_{\tor}(L/K). \]	
	The result then follows from the definition of $\psi^{\mathrm{ab}}_{L/K}$. 
\end{proof}

In the classical case, the functions we defined in fact coincide with the classical $\varphi$ and $\psi$-functions, as is shown in the following theorem.
\begin{theorem} \label{thm:psi}
	If $L/K$ is a finite Galois extension and $k$ is perfect, we have\[\psi_{L/K}=\psi^{\mathrm{ab}}_{L/K}=\psi^{\mathrm{AS}}_{L/K}. \]
\end{theorem}

\begin{proof} 
	From the definitions of the functions, we can assume that $k$ is algebraically closed.	We shall first show that  $\psi^{\mathrm{AS}}_{L/K}=\psi_{L/K}$. 
	To show  $\varphi_{L/K}(t)\leq \varphi_{L/K}^{\mathrm{AS}}(t)$, just observe that, if $L'/L$ is a finite Galois extension over $K$, then
	\[
	G(L'/L)^t\!=\!G(L'/L)_{\psi_{L'/K}\circ \varphi_{L/K}(t)} \!\subset\!  G(L'/K)_{\psi_{L'/K}\circ \varphi_{L/K}(t)}\!=\!G(L'/K)^{\varphi_{L/K}(t)}.
	\]
	Since the Abbes-Saito filtration is left continuous with rational jumps, it remains to show that  $\varphi_{L/K}(t)\geq \varphi_{L/K}^{\mathrm{AS}}(t)$ for $t\in\Q_{\geq 0}$. Let $K'$ be a finite Galois extension of $K$ that is tame and write $L'=LK'$. Since $L'/L$ and $K'/K$ are tame extensions, we have
	\begin{align*}
	\varphi_{L'/K'}(e(L'/L)t)&= e(K'/K)\varphi_{K'/K}\circ\varphi_{L'/K'}\circ\psi_{L'/L}(t)\\&=   e(K'/K)\varphi_{L'/K}\circ\psi_{L'/L}(t)\\&=e(K'/K)\varphi_{L/K}(t).
	\end{align*}
	
	From Serre's local class field theory for fields with algebraically closed residue field  (\cite{Serre1961}), for every $s\in\Z_{\geq 0}$, the maps
	\[
	\dfrac{\left(G_{L'}^{\mathrm{ab}}\right)^{\psi_{L'/K'}(s)}}{\left(G_{L'}^{\mathrm{ab}}\right)^{\psi_{L'/K'}(s)+1}}\to\dfrac{\left(G_{K'}^{\mathrm{ab}}\right)^{s}}{\left(G_{K'}^{\mathrm{ab}}\right)^{s+1}} 
	\]
	have images that are of finite index and nontrivial. Taking $K'$ such that $e(K'/K)\varphi_{L/K}(t)$ is an integer and setting $s=\varphi_{L'/K'}(e(L'/L)t)$, we see that
	the image of $(G_{L'}^{\mathrm{ab}})^{e(L'/L)t}=(G_{L}^{\mathrm{ab}})^{t}$ is not contained in $(G_{K'}^{\mathrm{ab}})^{\varphi_{L'/K'}(e(L'/L)t)+1}$. 
	
	Since $(G_{K'}^{\mathrm{ab}})^{e(K'/K)\varphi_{L/K}(t)+1}=(G_{K}^{\mathrm{ab}})^{\varphi_{L/K}(t)+\frac{1}{e(K'/K)}}$, we have that the image of $(G_{L}^{\mathrm{ab}})^{t}$ is not contained in
	$(G_{K}^{\mathrm{ab}})^{\varphi_{L/K}(t)+\frac{1}{e(K'/K)}}$. We can choose tame extensions with $e(K'/K)$ arbitrarily large, so we have that $\varphi_{L/K}(t)\geq \varphi_{L/K}^{\mathrm{AS}}(t)$. Hence $\psi_{L/K}^{\mathrm{AS}}=\psi_{L/K}$.

	Now we shall prove that $\psi_{L/K}^{\mathrm{ab}}=\psi_{L/K}$. 
	Let $K'/K$ be a finite, separable extension of complete discrete valuation fields that is tamely ramified and such that $e(LK'/L)t$ and $e(K'/K)\varphi_{L/K}(t)$ are integers. Write $L'=LK'$.   
	Observe that, taking into account that $K'/K$ and $L'/L$ are tamely ramified, we have
	\begin{align*}
	\psi_{L'/K'}(e(K'/K)\varphi_{L/K}(t))&=\psi_{L'/K'}\circ\psi_{K'/K}\circ\varphi_{L/K}(t)=\psi_{L'/K}\circ \varphi_{L/K}(t)\\&= \psi_{L'/L}\circ \psi_{L/K} \circ \varphi_{L/K}(t)= \psi_{L'/L}(t) = e(L'/L)t.
	\end{align*}
	Let  \[\chi \in F_{e(K'/K)\varphi_{L/K}(t)}H^1(K').\]
	Denote by $\chi_{L'}$ its image in $H^1(L')$. Using the same argument as before we see that  $\chi_{L'}\in F_{e(L'/L)t}H^1(L')$,
	so $\varphi_{L/K}(t)\leq \varphi_{L/K}^{\mathrm{ab}}(t)$. Now, if $s= \varphi_{L/K}(t)+ \frac{1}{e(K'/K)}$, then
	\[
	F_{e(K'/K)s}H^1(K')= F_{e(K'/K)\varphi_{L/K}(t)+1}H^1(K').
	\] 
	Since $F_{e(L'/L)t}H^1(L')$ does not contain  the image of $F_{e(K'/K)\varphi_{L/K}(t)+1}H^1(K')$, we have that $s>\varphi_{L/K}^{\mathrm{ab}}(t)$. Since we can take extensions $K'/K$ with arbitrarily large $e(K'/K)$, we get that $\varphi_{L/K}=\varphi_{L/K}^{\mathrm{ab}}$. Thus $\psi_{L/K}^{\mathrm{ab}}=\psi_{L/K}$. 
\end{proof}

The properties we proved and Theorem \ref{thm:psi} give  evidence that the above defined functions $\psi^{\mathrm{ab}}_{L/K}$ and  $\psi^{\mathrm{AS}}_{L/K}$  are good generalizations of the classical $\psi$-function. We can conjecture:
\begin{conjecture} \label{conj:psi}
	Let $L/K$ be an extension of complete discrete valuation fields. Assume that $k$ is perfect of characteristic $p>0$. Then 
	\[ \psi^{\mathrm{ab}}_{L/K} = \psi^{\mathrm{AS}}_{L/K}.\]
\end{conjecture}
\begin{conjecture}
		Let $L/K$ be an extension of complete discrete valuation fields. Assume that $k$ is perfect of characteristic $p>0$. Then $\psi^{\mathrm{ab}}_{L/K}$ and $\psi^{\mathrm{AS}}_{L/K}$ are continuous, piecewise linear, increasing, and convex.
\end{conjecture}

\begin{ack}
	I would like to express my sincere gratitude to my advisor, Professor Kazuya Kato, for his kind advice and feedback, and for our many discussions which always enlighten me.
	
	I would also like to thank the anonymous referee for their careful reading, and for providing comments and suggestions that helped improve this paper.
\end{ack}

\bibliographystyle{amsplain}
\bibliography{bibfile}

\providecommand{\bysame}{\leavevmode\hbox to3em{\hrulefill}\thinspace}
\providecommand{\MR}{\relax\ifhmode\unskip\space\fi MR }
\providecommand{\MRhref}[2]{%
  \href{http://www.ams.org/mathscinet-getitem?mr=#1}{#2}
}
\providecommand{\href}[2]{#2}
\begin{thebibliography}{10}

\bibitem{abbessaito}
Ahmed Abbes and Takeshi Saito, \emph{Ramification of local fields with
  imperfect residue fields}, Amer. J. Math. \textbf{124} (2002), no.~5,
  879\textendash 920.

\bibitem{Brylinski1983}
Jean-Luc Brylinski, \emph{Th\'{e}orie du corps de classes de {K}ato et
  rev\^{e}tements ab\'{e}liens de surfaces}, Ann. Inst. Fourier (Grenoble)
  \textbf{33} (1983), no.~3, 23--38.

\bibitem{kato1982galois}
Kazuya Kato, \emph{Galois cohomology of complete discrete valuation fields},
  (1982), 215--238, Lecture {N}otes in {M}athematics, {V}ol. 967.

\bibitem{kato1983residue}
\bysame, \emph{Residue homomorphisms in {M}ilnor {$K$}-theory}, Galois groups
  and their representations, Adv. Stud. Pure Math. \textbf{2} (1983), 153--172.

\bibitem{kato1989swan}
\bysame, \emph{Swan conductors for characters of degree one in the imperfect
  residue field case}, Contemp. Math. \textbf{83} (1989), 101\textendash131.

\bibitem{kato2010modulus}
Kazuya Kato and Henrik Russell, \emph{Modulus of a rational map into a
  commutative algebraic group}, Kyoto J. Math. \textbf{50} (2010), no.~3,
  607--622.

\bibitem{kurihara1987two}
Masato Kurihara, \emph{On two types of complete discrete valuation fields},
  Comp. Math \textbf{63} (1987), 237--257.

\bibitem{kurihara1998exponential}
\bysame, \emph{The exponential homomorphisms for the {M}ilnor {$K$}-groups and
  an explicit reciprocity law}, J. Reine Angew. Math. \textbf{498} (1998),
  201--221.

\bibitem{isabel1}
Isabel Leal, \emph{On the ramification of \'{e}tale cohomology groups}, J.
  Reine Angew. Math. (2016), DOI:
  \href{http://doi.org/10.1515/crelle-2016-0035}{http://doi.org/10.1515/crelle-2016-0035}.

\bibitem{matsuda1997swan}
Shigeki Matsuda, \emph{On the swan conductor in positive characteristic},
  American Journal of Mathematics \textbf{119} (1997), no.~4, 705--739.

\bibitem{morrow2010explicit}
Matthew Morrow, \emph{An explicit approach to residues on and dualizing sheaves
  of arithmetic surfaces}, New York J. Math. \textbf{16} (2010), 575--627.

\bibitem{Serre1961}
Jean-Pierre Serre, \emph{Sur les corps locaux à corps résiduel
  algébriquement clos}, Bull. Soc. Math. France \textbf{89} (1961), 105--154.

\bibitem{serre1979local}
\bysame, \emph{Local fields, volume 67 of {G}raduate {T}exts in {M}athematics},
  1979.

\bibitem{yatagawa2016equality}
Yuri Yatagawa, \emph{Equality of two non-logarithmic ramification filtrations
  of abelianized galois group in positive characteristic}, Documenta Math.
  \textbf{22} (2017), 917--952.

\bibitem{zhukov2000}
Igor Zhukov, \emph{Higher dimensional local fields}, Invitation to higher local
  fields (Munster, 1999) Geom. Topol. Monogr. 3, 2000, pp.~5--18.

\end{thebibliography}

\end{document}